\documentclass[11pt,a4paper,reqno]{amsart}
\usepackage{amsmath,amsthm,amsfonts,amssymb,amscd,amstext}
\usepackage[latin1]{inputenc}
\usepackage[dvips]{graphicx}
\usepackage{psfrag}
\usepackage{euscript}
\usepackage{a4wide}
\usepackage{mathrsfs}
\usepackage{xcolor}
\usepackage{mathtools}
\usepackage[colorlinks=true,linkcolor=blue, urlcolor=red, citecolor=blue, backref]{hyperref}	

\numberwithin{equation}{section}

\theoremstyle{plain}
\newtheorem{maintheorem}{Theorem}

\newtheorem{maincorollary}[maintheorem]{Corollary}

\newtheorem{theorem}{Theorem}[section]
\newtheorem{proposition}[theorem]{Proposition}
\newtheorem{corollary}[theorem]{Corollary}
\newtheorem{lemma}[theorem]{Lemma}

\theoremstyle{definition}
\newtheorem{remark}[theorem]{Remark}
\newtheorem{definition}{Definition}
\newtheorem{example}{Example}

\newcommand{\vep}{\varepsilon}

\renewcommand{\epsilon}{\varepsilon}

\newcommand{\intt}{\operatorname{int}}

\newcommand{\Trans}{\operatorname{Trans}}

\begin{document}


\title[Sensitivity and historic behavior on Baire spaces]{Sensitivity and historic behavior for continuous \\ maps on Baire metric spaces}

\author{M. Carvalho}
\author{V. Coelho}
\author{L. Salgado}
\author{P. Varandas}

\address{Maria Carvalho, CMUP \& Departamento de Matem\'atica, Faculdade de Ci\^encias da Universidade do Porto \\
Rua do Campo Alegre 687, 4169-007 Porto, Portugal.}
\email{mpcarval@fc.up.pt}

\address{Vin\'icius Coelho, Universidade Federal do Oeste da Bahia, Centro Multidisciplinar de Bom Jesus da Lapa\\
Av. Manoel Novais, 1064, Centro, 47600-000 - Bom Jesus da Lapa-BA-Brazil}
\email{viniciuscs@ufob.edu.br}

\address{Luciana Salgado, Universidade Federal do Rio de Janeiro, Instituto de
   Matem\'atica\\
   Avenida Athos da Silveira Ramos 149 Cidade Universit\'aria, P.O. Box 68530,
   21941-909 Rio de Janeiro-RJ-Brazil }
 \email{lsalgado@im.ufjr.br, lucianasalgado@ufrj.br}

\address{Paulo Varandas, CMUP, Universidade do Porto \& Departamento de Matem\'atica, Universidade Federal da Bahia\\
Av. Ademar de Barros s/n, 40170-110 Salvador, Brazil}
\email{paulo.varandas@ufba.br}

\subjclass[2010]{Primary: 37A30, 37C10, 37C40, 37D20.}

\keywords{Baire metric space; Historic behavior; Irregular set; Transitivity;
Sensitivity, Group and semigroup actions}


\begin{abstract}
We introduce a notion of sensitivity, with respect to a continuous bounded observable, which provides a sufficient condition for a continuous map, acting on a Baire metric space, to exhibit a Baire generic subset of points with historic behavior.
The applications of this criterion recover, and extend, several known theorems on the genericity of the irregular set, besides yielding a number of new results, including information on the irregular set of geodesic flows, in both negative and non-positive curvature, and semigroup actions.
\end{abstract}


\date{\today}
\maketitle


\section{Introduction} \label{sec:statement-result1}

\subsection{Historic behavior}

In what follows, we shall write $X$ to denote a compact metric space and $Y$ will stand for an arbitrary metric space. Given such a space $Y$ and $A \subset Y$, denote by $A'\subset S$ the set of non-isolated accumulation points in $A$, that is, $y \in A'$ if and only if $y$ belongs to the closure $\overline{A \setminus \{y\}}$.
Let $C(Y, \mathbb{R})$ be the set of real valued, continuous maps on $Y$ and $C^{b}(Y,\mathbb{R})$ be its subset of bounded elements endowed with the supremum norm $\|\cdot\|_\infty$.

Given a Baire metric space $(Y, d)$, a continuous map $T : Y \to Y$ and $\varphi \in C^b(Y,\mathbb R)$, the set of \emph{$(T,\varphi)$-irregular points}, or points with historic behavior, is defined by
\begin{center}
${\mathcal I}(T,\varphi) =  \Big\{ x \in X : \Big(\lim \limits_{n \,\to \,+\infty} \frac{1}{n} \sum \limits_{j=0}^{n-1}  \varphi (T^{j}(x)) \Big)_{n \,\in \,\mathbb{N}} $ does not converge$\Big\}$.
\end{center}
The Birkhoff's ergodic theorem ensures that, for any Borel $T$-invariant probability measure $\mu$ and every $\mu$-integrable observable $\varphi: Y \to \mathbb R$, the sequence of averages $\big(\frac1n \sum_{j=0}^{n-1}\varphi (T^j(y))\big)_{n \, \in \, \mathbb{N}}$ converges at $\mu$-almost every point in $Y$.
So, the set of $(T,\varphi)$-irregular points is negligible with respect to any $T$-invariant probability measure. In the last decades, though, there has been an intense study concerning the set of points for which Ces\`aro averages do not converge. Contrary to the previous measure-theoretical description, the set of the irregular points may be Baire generic and, moreover, have full topological pressure, full metric mean dimension or full Hausdorff dimension (see \cite{AP19, BKNRS, BLV, BS, LV, PS07, TV17}).
In \cite{CV21}, the first and the fourth named authors obtained a simple and unifying criterion, using first integrals, to guarantee that ${\mathcal I}(T,\varphi)$ is Baire generic whenever $T:X \to X$ is a continuous dynamics acting on a compact metric space $X$. More precisely, given $\varphi \in C(X, \mathbb{R})$, consider the map $L_{\varphi}: X \to \mathbb R$ defined by
\begin{equation}\label{eq:firstintegral}
x \in X \quad \mapsto \quad L_{\varphi}(x) = \limsup \limits_{n\to\infty} \frac{1}{n} \sum \limits_{j=0}^{n-1} \varphi \circ T^{j} (x).
\end{equation}
This is a first integral with respect to the map $T$, that is, $L_\varphi \circ T = L_\varphi$. The existence of dense sets of discontinuity points for this first integral turns out to be a sufficient condition for the genericity of the historic behavior.

\begin{theorem}\cite[Theorem A]{CV21} \label{theoremCV}
Let $(X, d)$ be a compact metric space, $T : X \to X$ be a continuous map and $\varphi : X \to \mathbb{R}$ be a continuous observable. Assume that there exist two dense subsets $A,B\subset X$ such that the restrictions of $L_{\varphi}$ to $A$ and to $B$ are constant, though the value at $A$ is different from the one at $B$. Then ${\mathcal I}(T,\varphi)$ is a Baire generic subset of $X$.
\end{theorem}

The assumptions of the previous theorem are satisfied by a vast class of continuous maps on compact metric spaces, including minimal non-uniquely ergodic homeomorphisms, non-trivial homoclinic classes, continuous maps with the specification property, Viana maps and partially hyperbolic diffeomorphisms (cf. \cite{CV21}).

In this work we establish a criterion with a wider scope than the one of Theorem~\ref{theoremCV}. It applies to Baire metric spaces and general sequences of bounded continuous real-valued maps, rather than just Ces\`aro averages, subject to a weaker requirement than the one demanded in the previous theorem. In particular, one obtains new results on the irregular set of several classes of maps and flows, which comprise geodesic flows on certain non-compact Riemannian manifolds, countable Markov shifts and endomorphisms with two physical measures exhibiting intermingled basins of attraction. Regarding semigroup actions, we note that irregular points for group actions with respect to Ces\`aro averages were first studied in \cite{FV}. We will provide  additional information on the Baire genericity of irregular sets for averages that take into account the group structure.
We refer the reader to Sections~\ref{sec:examples} and \ref{sec:applications} for the precise statements.

In the next subsections we will state our main definitions and results.


\subsection{Sensitivity and genericity of historic behavior}

Given a metric space $(Y,d)$, a sequence $\Phi = (\varphi_n)_{n \, \in \mathbb{N}} \in C^b(Y,\mathbb R)^{\mathbb N}$ and $y \in Y$, let $W_{\Phi}(y) :=  \big\{\varphi_{n}(y):  \,\, n \in \mathbb{N} \big\}'$ denote the set of accumulation points of the sequence $(\varphi_{n}(y))_{n\, \in \,\mathbb{N}}.$
The next notion is inspired by the concept of sensitivity to initial conditions.

\begin{definition}\label{defTphisensitive}
\emph{Let $(Y,d)$ be a metric space and $\Phi \in C^b(Y,\mathbb R)^{\mathbb N}$. We say that $Y$ is $\Phi$-\emph{sensitive} (or sensitive with respect to the sequence $\Phi$) if there exist dense subsets $A,\,B \subset Y$, where $B$ can be equal to $A$, and $\varepsilon > 0$ such that for any $(a,b) \in A \times B$ one has
$$\sup_{r\,\in \,W_{\Phi}(a), \,\, s \,\in \,W_{\Phi}(b)} \,\,|r-s| > \varepsilon.$$
In the particular case of a sequence $\Phi$ of Ces\`aro averages $(\varphi_{n})_{n\, \in \,\mathbb{N}} = \big(\frac{1}{n} \,\sum_{j=0}^{n-1} \,\varphi\circ T^{j}\big)_{n \, \in \, \mathbb{N}}$ associated to a potential $\varphi \in C^b(Y,\mathbb R)$ and a continuous map $T: Y \to Y$, we say that $Y$ is $(T,\varphi)$-\emph{sensitive} if the space $Y$ is $\Phi$-\emph{sensitive}, and write $W_{\varphi}$ instead of $W_{\Phi}$.}
\end{definition}

We refer the reader to Example~\ref{ex:exemplo1} as an illustration of this definition. We observe that being $(T,\varphi)$-sensitive is a direct consequence of the assumption on $L_\varphi$ stated in Theorem~\ref{theoremCV}, though it may be strictly weaker (cf. Example~\ref{ex:Bowen}).
Our first result concerns $\Phi$-sensitive sequences and strengths Theorem~\ref{theoremCV}.

\begin{maintheorem}\label{CV_mod}
Let $(Y, d)$ be a Baire metric space and $\Phi \in C^b(Y,\mathbb R)^{\mathbb N}$ be a sequence of continuous, bounded maps such that $\limsup_{n\,\to\,+\infty} \,\|\varphi_n\|_\infty < +\infty$. If $Y$ is $\Phi$-sensitive then the set
$${\mathcal I}(\Phi):= \Big\{y \in Y \colon \,\lim_{n\,\to \,+\infty} \,\varphi_n(y) \, \text{does not exist} \,\Big\}$$
is a Baire generic subset of $Y$. In particular, if $T: Y \to Y$ is a continuous map, $\varphi$ belongs to $C^b(Y,\mathbb R)$ and $Y$ is $(T,\varphi)$-sensitive, then ${\mathcal I}(T,\varphi)$ is a Baire generic subset of $Y$.
\end{maintheorem}


\subsection{Irregular points for continuous maps on Baire metric spaces with dense orbits}
In this subsection, building over \cite{Dow,HLT21}, we will discuss the relation between transitivity, existence of dense orbits and the size of the set of irregular points for continuous maps on Baire metric spaces.

\begin{definition}\label{def:transitivity}
Given a continuous map $T: Y \to Y$ on a Baire metric space $(Y,d)$, one says that:
\begin{itemize}
\item  $T$ is \textit{transitive} if for every non-empty open sets $U,\,V \subset Y$ there exists $n  \in \mathbb{N}$ such that $U \cap T^{-n}(V) \neq \emptyset$.
\smallskip
\item $T$ is \textit{strongly transitive} if $\bigcup_{n \, \geqslant \,0}\, T^n(U) = Y$ for every non-empty open set $U \subset Y$.
\smallskip
\item $T$ has a \textit{dense orbit} if there is $y \in Y$ such that $\{T^{j}(y):  \,j \, \in \, \mathbb{N}\cup \{0\}\}$ is dense in $Y$.
\end{itemize}
\end{definition}

It is worthwhile observing that,
if the metric space is compact and has no isolated points, then the map is transitive if and only if it has a dense orbit (see \cite[Theorem~1.4]{AC12} and Example~\ref{ex:exemplo2}).

Denote by $\Trans (Y,T)$ the set $\{y \in Y \colon \,\text{ the orbit of $y$ by $T$ is dense in } Y\}$ and take on the following notation:
\begin{align*}
& \mathcal{H}(Y,T) \coloneqq  \{\varphi \in C^{b}(Y,\mathbb{R})\colon \,{\mathcal I}(T,\varphi) \neq \emptyset\}
\\
& \mathcal{D}(Y,T) \coloneqq \{\varphi \in C^{b}(Y,\mathbb{R})\colon \,{\mathcal I}(T,\varphi)\text{ is dense in } Y\}
\\
& \mathcal{R}(Y,T) \coloneqq \{\varphi \in C^{b}(Y,\mathbb{R})\colon \,{\mathcal I}(T,\varphi)\text{ is Baire generic in } Y\}.
\end{align*}
The set of \emph{completely irregular points} with respect to $T$ is precisely the intersection
$$\bigcap\limits_{\varphi \,\in \,\mathcal{H}(Y,T)} \,{\mathcal I}(T,\varphi).$$

\noindent Clearly $\mathcal{R}(Y,T) \subset \mathcal{D}(Y,T) \subset \mathcal{H}(Y,T)$. Inspired by \cite{HLT21}, we aim at finding sufficient conditions on $(Y,T)$ and $\varphi$ under which the sets $\mathcal{R}(Y,T)$ and $\mathcal{D}(Y,T)$ coincide. We note that the set $\mathcal{H}(Y,T)$ may be uncountable and that Tian proved in (cf. \cite[Theorem~2.1]{Tian}) that, if $Y$ is compact and $T$ has the almost-product and uniform separation properties, then the set of completely irregular points is either empty or carries full topological entropy. More recently, in \cite{HLT21}, Hou, Lin and Tian showed that, for each transitive continuous map on a compact metric space $Y$, either every point with dense orbit is contained in the basin of attraction of an invariant probability measure (so it is regular with respect to any continuous potential), or irregular behavior occurs on $\Trans (Y,T)$ and the irregular set is Baire generic for every \emph{typical} $\varphi \in C(X, \mathbb R)$. The next consequence of Theorem~\ref{CV_mod} generalizes this information.

\begin{maincorollary}\label{maincorollary_auxiliar}
Let $(Y, d)$ be a Baire metric space and $T : Y \to Y$ be a continuous map such that $\Trans (Y,T) \neq \emptyset$. Then:
\begin{itemize}
\item[($i$)] When $(Y,d)$ has an isolated point,
$$\mathcal{D}(Y,T) \,\neq \,\emptyset \quad \quad \Leftrightarrow \quad \quad \Trans (Y,T) \,\subseteq \bigcap\limits_{\varphi \,\in \,\mathcal{D}(Y,T)} \,{\mathcal I}(T,\varphi).$$
In particular, if $Y$ has an isolated point and $\mathcal{D}(Y,T) \neq \emptyset$, then $\mathcal{R}(Y,T) = \mathcal{D}(Y,T)$.
\smallskip
\item[($ii$)] If $\mathfrak{F} \subset C^{b}(Y, \mathbb{R})$ and $\bigcap\limits_{\varphi \,\in \,\mathfrak{F}} \,{\mathcal I}(T,\varphi)$ is Baire generic, then
$$\Trans (Y,T) \,\cap \, \bigcap\limits_{\varphi \,\in \,\mathfrak{F}} \,{\mathcal I}(T,\varphi) \neq \emptyset.$$
\end{itemize}
\end{maincorollary}


\subsection{Oscillation of the time averages}

Define for any $\varphi \in C^{b}(Y,\mathbb{R})$ and $y \in Y$,
$$\ell_\varphi(y):= \,\liminf\limits_{n\,\to\,+\infty}\,\frac{1}{n}\,\sum_{j=0}^{n-1}\,\varphi(T^{j}y) \quad \quad \text{ and } \quad \quad L_\varphi(y):= \,\limsup\limits_{n\,\to\,+\infty}\,\frac{1}{n}\,\sum_{j=0}^{n-1}\,\varphi (T^{j}y)$$
and consider
\begin{equation}\label{eq:bounds-transitive-points}
\ell_\varphi^{*} =\, \inf\limits_{y\,\in\, \Trans(Y,T)} \,\ell_\varphi(y)  \quad \quad \text{ and } \quad \quad L_\varphi^{*} =\,\sup\limits_{y\,\in\, \Trans(Y,T)} \, L_\varphi(y).
\end{equation}
For each $\alpha \leqslant \beta$, take the sets
$$I_{\varphi}[\alpha,\,\beta] \,:= \,\Big\{y \in Y \colon \,\ell_\varphi(y) = \alpha \text{ and } \beta = L_\varphi(y) \Big\}$$
and
$$\widehat{I_{\varphi}[\alpha,\,\beta]}\,:=\,\Big\{y \in Y \colon \,\ell_\varphi(y) \leqslant  \alpha \text{ and } \beta \leqslant L_\varphi(y) \Big\}.$$

\noindent The next corollary estimates the topological size of the previous level sets for dynamics with dense orbits.

\begin{maincorollary}\label{maincorollary_sieve}
Let $(Y, d)$ be a Baire metric space and $T : Y \to Y$ be a continuous map with a dense orbit. Given $\varphi \in C^b(Y,\mathbb R)$, one has:
\begin{itemize}
\item[$(i)$] $I_{\varphi}[\ell_\varphi^{*},\,L_\varphi^{*}]$ is a Baire generic subset of $Y$. In particular, ${\mathcal I}(T,\varphi)$ is either Baire generic or meagre in $Y$.
\smallskip
\item[$(ii)$] ${\mathcal I}(T,\varphi)$ is a meagre subset of $Y$ if and only if there exists $C_{\varphi} \in \mathbb{R}$ such that
$$
I_{\varphi}[C_{\varphi},\,C_{\varphi}] \,= \, \Big\{y \in Y \colon \,\lim\limits_{n\, \to\, +\infty} \,\frac{1}{n} \,\sum\limits_{j=0}^{n-1} \,\varphi (T^{j}(y)) = C_{\varphi} \Big\}
$$
is a Baire generic set containing $\Trans (Y,T)$.
\end{itemize}
\end{maincorollary}

Consequently, if $(Y, d)$ is a Baire metric space, $T : Y \to Y$ is a continuous map with a dense orbit and $\bigcup_{\varphi \,\in \,C^{b}(Y,T)}\, {\mathcal I}(T,\varphi)$ is not Baire generic, then $Y \setminus {\mathcal I}(T,\varphi)$ is Baire generic for every $\varphi \in C^{b}(Y,\mathbb{R})$, and there exists a linear functional
$\mathcal{F}: C^{b}(Y,\mathbb{R}) \to \mathbb{R}$ such that
$$\Trans (Y,T) \subset \Big\{y \in Y: \lim_{n \,\to \, +\infty} \,\frac{1}{n} \,\sum_{j=0}^{n-1} \,\varphi (T^{j}(y)) \,= \,\mathcal{F}(\varphi) \quad \quad \forall\, \varphi \in C^{b}(Y,\mathbb{R}) \Big\}.$$
Note that, if $Y$ is compact, $\mathcal{F}$ is represented by the space of Borel invariant measures with the weak$^*$-topology. More generally, if $Y$ is a metric space, then the dual of $C^{b}(Y,\mathbb{R})$ is represented by the regular, bounded, finitely additive set functions with the norm of total variation (cf. \cite[Theorem 2, IV.6.2]{DS}).

\begin{remark}\label{rmk:minimal} Corollary~\ref{maincorollary_sieve} implies that, if $(Y,d)$ is a Baire metric space and $T: Y \to Y$ is a continuous minimal map (that is, $\Trans (Y,T) = Y$), then for every $\varphi \in C^b(X, \mathbb{R})$ either ${\mathcal I}(T,\varphi) = \emptyset$ or ${\mathcal I}(T,\varphi)$ is Baire generic.
\end{remark}

\begin{remark} From \cite[Theorem B]{HLT21}, we know that, in the context of compact metric spaces and transitive maps, the set ${\mathcal I}(T,\varphi)$ is either meagre or Baire generic for every $\varphi \in C(X, \mathbb R)$. It is unknown whether there exist a Baire metric space $(Y, d)$, a continuous map $T: Y \to Y$ and $\varphi \in C^b(Y,\mathbb R)$ such that $T$ has a dense orbit and ${\mathcal I}(T,\varphi)$ is a non-empty meagre set.
\end{remark}

The next corollary is a counterpart of  \cite[Lemma 3.1]{HLT21} for continuous maps on Baire metric spaces.

\begin{maincorollary}\label{maincorollary_equivalence}
Let $(Y, d)$ be a Baire metric space and $T : Y \to Y$ be a continuous map with a dense orbit, and take $\varphi \in C^b(Y,\mathbb R)$. The following conditions are equivalent:
\begin{itemize}
\item[$(i)$] $Y$ is $(T,\varphi)$-sensitive.
\smallskip
\item[$(ii)$] ${\mathcal I}(T,\varphi)$ is Baire generic in $X$.
\smallskip
\item[$(iii)$] $\Trans (Y,T) \cap {\mathcal I}(T,\varphi) \neq \emptyset$.
\end{itemize}
\end{maincorollary}

\begin{remark}
Corollary~\ref{maincorollary_equivalence} implies that $\mathcal{R}(Y,T) = \big\{\varphi \in C^{b}(Y,\mathbb{R}) \colon \,\ell_\varphi^{*} \,<\,  L_\varphi^{*}\big\}.$
We also note that one has $(i) \Rightarrow (ii)$ in Corollary~\ref{maincorollary_equivalence} even without assuming the existence of dense orbits.
\end{remark}

\begin{remark} An immediate consequence of the proof of Corollary~\ref{maincorollary_equivalence} is that, within the setting of continuous maps $T:Y \to Y$ acting on a Baire metric space $Y$ and having a dense orbit, the Definition~\ref{defTphisensitive} of $(T,\varphi)$-sensitivity is equivalent to the following statement: \emph{There exist a dense set $A \subset Y$ and $\varepsilon > 0$ such that for any $a \in A$ one has
$$\sup_{r\,\in \,W_{\varphi}(a), \,\, s \,\in \,W_{\varphi}(a)} \,\,|r-s| > \varepsilon.$$}
Indeed, the latter statement clearly implies Definition~\ref{defTphisensitive}. Conversely, if $X$ is $(T,\varphi)$-sensitive, then we may take $A = \{T^j(x_0) \colon \, j \in \mathbb{N}\,\cup\,\{0\}\}$ equal to a dense orbit contained in ${\mathcal I}(T,\varphi)$, whose existence is guaranteed by item $(iii)$ of Corollary~\ref{maincorollary_equivalence}. Given two distinct accumulation points $r, \,s$ in $W_\varphi(x_0)$ and $\varepsilon = |r-s| > 0$, then for every $a, \, b \in A$ one has $W_\varphi(a) = W_\varphi(b) = W_\varphi(x_0)$ and $\sup_{r_a\,\in \,W_{\Phi}(a), \,\, r_b \,\in \,W_{\Phi}(b)} \, \geqslant \,|r-s| > \varepsilon.$
\end{remark}

The remainder of the paper is organized as follows. In Section~\ref{se:compact}, we convey the previous results to the particular case of continuous dynamics acting on compact metric spaces. The aforementioned results are then proved in the ensuing sections, where we also discuss their scope and compare them with properties established in other references. In Section~\ref{sec:examples} we test our assumptions on some examples and in Section~\ref{sec:applications} we provide some applications, namely within the settings of semigroup actions and geodesic flows on non-compact manifolds.


\section{Irregular points for continuous maps on compact metric spaces}\label{se:compact}

Suppose now that $(X,d)$ is a compact metric space. Let $\mathcal{P}(X)$ stand for the Borel probability measures on $X$ with the weak$^*$-topology and consider a continuous map  $T: X \to X$ with a dense orbit. For every $x \in X$, let $\delta_{x}$ be the Dirac measure supported at $x$ and denote the set of accumulation points in $\mathcal{P}(X)$ of the sequence of empirical measures $\big(\frac{1}{n}\,\sum_{j=0}^{n-1}\,\delta_{T^{j}x}\big)_{n \, \in \, \mathbb{N}}$ by $V_T(x)$. Our next result is a refinement of \cite[Theorem A]{HLT21}, since it imparts new information about the irregular set, besides establishing a generalization of \cite[Theorem B]{HLT21} inasmuch as it does not require any assumption about isolated points.

\begin{corollary}\label{maincorollary_compact}
Let $(X, d)$ be a compact metric space and $T : X \to X$ be a continuous map with a dense orbit. Then:

\begin{itemize}
\item[$(a)$]  $X_{\Delta} \coloneqq \,\Big\{x \in X \colon \bigcup_{t \,\in\, \Trans(X,T)} \,V_{T}(t) \,\subseteq \,V_{T}(x)\Big\}$ is Baire generic in $X$. Moreover, $X_{\Delta} \,\subseteq\, \bigcap_{\varphi \,\in \,C(X,\mathbb{R})}\, \widehat{I_{\varphi}[\ell_\varphi^{*},\,L_\varphi^{*}]}$, so the latter set is Baire generic as well.
\medskip
\item[$(b)$] $\bigcup_{\varphi \,\in \,C(X,\mathbb{R})}\, {\mathcal I}(T,\varphi)$ is either Baire generic or meagre. If it is meagre, there exists a Borel $T$-invariant measure $\mu$ such that
$$\Trans (X,T) \, \subset \, \Big\{x \in X: \,\lim_{n \,\to \,+\infty} \,\frac{1}{n} \,\sum_{j=0}^{n-1} \,\varphi (T^{j}(x)) = \int \varphi \,d\mu \quad  \forall \,\varphi \in C(X,\mathbb{R})\Big\}.$$
\item[$(c)$] $\bigcap_{\varphi \,\in\, \mathcal{H}(X,T)}\, {\mathcal I}(T,\varphi)$ is either Baire generic or meagre. In addition,
$$\bigcap_{\varphi \,\in\, \mathcal{H}(X,T)}\, {\mathcal I}(T,\varphi) \text{ is Baire generic } \,\, \Leftrightarrow \,\, \Big\{\varphi \in C(X,\mathbb{R})\colon \, {\mathcal I}(T,\varphi) \neq \emptyset \text{ and } \ell_\varphi^{*} = L_\varphi^{*}\Big\} \,= \, \emptyset.$$
\end{itemize}
\end{corollary}

\smallskip

\begin{remark}\label{remark_minimal} It follows from Corollary \ref{maincorollary_compact} that, if $(X, d)$ is a compact metric space and $T : X \to X$ is a continuous minimal map, then either $T$ is uniquely ergodic or the set $\bigcap_{\varphi \,\in\, \mathcal{H}(X,T)} \, {\mathcal I}(T,\varphi)$ is Baire generic.
\end{remark}


We may ask whether the notion of $(T,\varphi)$-sensitivity (cf. Definition~\ref{defTphisensitive}) is somehow related to the classical concepts of sensitivity to initial conditions and expansiveness. Let us recall these two notions.

\begin{definition}\label{def:sensitive}
Let $T: X \to X$ be a continuous map on a compact metric space $(X,d)$. We say that:
\begin{itemize}
\item $T$ has \emph{sensitivity to initial conditions} if there exists $\varepsilon > 0$ such that, for every $x \in X$ and any $\delta > 0$, there is $z \in B(x,\delta)$ satisfying
$$\sup_{n\,\in\,\mathbb N} \, d(T^{n}(x), T^{n}(z)) > \varepsilon.$$
\item $T$ is (positively) \emph{expansive} if there exists $\varepsilon > 0$ such that, for any points $x,z \in X$ with $x \neq z$, one has
$$\sup_{n\,\in\,\mathbb N} \,d(T^{n}(x), T^{n}(z)) > \varepsilon.$$
\end{itemize}
\end{definition}

It is clear that an expansive map has sensitivity to initial conditions. Our next result shows that the condition of $(T,\varphi)$-sensitivity often implies sensitivity to initial conditions.

\begin{theorem}\label{sensitive}
Let $X$ be a compact metric space, $T : X \to X$ be a continuous map and $\varphi \in C(X,\mathbb R)$. If $X$ is $(T,\varphi)$-sensitive then either $X$ has sensitivity to initial conditions or ${\mathcal I}(T,\varphi)$ has non-empty interior. In particular, if
$T$ has a dense set of periodic orbits and $X$ is $(T,\varphi)$-sensitive, then $X$ has sensitivity to initial conditions.
\end{theorem}

The previous discussion together with well known examples yield the following scheme of connections:
\begin{equation}\label{eq:diagram}
\begin{array}{ccc}
\text{$T$ is strongly transitive} 						&  &   \\
\qquad \text{ and has dense periodic orbits} 						&  &  \\
	 \Downarrow					&  &  \\
\text{$T$ has dense periodic orbits} 						&  &   \\
\text{and dense pre-orbits} 						&  & \text{Expansiveness}  \\
	 \Downarrow					&  &  \Downarrow \; \; \not\Uparrow\\
\exists \,\varphi\in C(X, \mathbb R) \colon \, X \text{ is } (T,\varphi)-{ sensitive }  & \quad \Rightarrow  \quad &  \text{Sensitivity to initial conditions}.  \\
\smallskip
\text{and} \; \text{int}({\mathcal I}(T,\varphi))=\emptyset  &  &
\end{array}
\end{equation}

When $X$ is a compact topological
manifold there is a link between expansiveness and sensitivity regarding a well chosen continuous map $\varphi: X \to \mathbb R$. Indeed, Coven and Reddy (cf. \cite{CoRe}) proved that, if $T: X \to X$ is a continuous and expansive map acting on compact topological
manifold, then there exists a metric $\tilde d$ compatible with the topology of $(X,d)$ such that $T: (X,\tilde d) \to (X,\tilde d)$ is a Ruelle-expanding map: there are constants $\lambda > 1$ and $\delta_0 > 0 $ such that, for all $x, y, z\in X$, one has
\begin{itemize}
\item $\tilde d(T(x),T(y)) \geqslant \lambda \tilde d(x,y)$ whenever $\tilde d(x,y) < \delta_0$;
\smallskip
\item $B(x,\delta_0) \cap T^{-1}(\{z\})$ is a singleton whenever $\tilde d(T(x),z) < \delta_0$.
\end{itemize}
In particular, if $X$ is connected then $T$ is topologically mixing. Moreover, as Ruelle-expanding maps admit finite Markov partitions and are semiconjugate to subshifts of finite type, one can choose $\varphi \in C(X, \mathbb R)$ such that ${\mathcal I}(T,\varphi)$ is a Baire generic subset of $X$. Besides, the interior of ${\mathcal I}(T,\varphi)$ is empty by the denseness of the set of periodic points of $T$. Therefore, by Corollary~\ref{maincorollary_equivalence}:

\begin{corollary}\label{cor:Riemann}
Let $(X,d)$ be a compact connected
topological manifold and $T: X\to X$ be a continuous and expansive map. Then there exists $\varphi\in C(X, \mathbb R)$ such that $X$ is $(T,\varphi)$-sensitive.
\end{corollary}

It is still an open question whether for each expansive map $T$ on a compact metric space $X$ there exists $\varphi\in C(X, \mathbb R)$ such that $X$ is $(T,\varphi)$-sensitive. Although we have no examples, it is likely to exist continuous maps on compact metric spaces which have sensitivity to initial conditions but for which $X$ is not $(T,\varphi)$-sensitive for every $\varphi\in C(X, \mathbb R)$.

A further consequence of our results concerns the irregular sets for
continuous maps satisfying the strong transitivity condition (see Definition~\ref{def:transitivity}).
 equivalent to say that $T$ is minimal.

\begin{corollary}\label{cor:minimalFV}
Let $T: X\to X$ be a continuous map on a compact metric space $X$. If $T$ is strongly transitive and $\varphi \in C(X,\mathbb R)$, then either
$\mathcal I(T,\varphi)=\emptyset$ or $\mathcal I(T,\varphi)$ is a Baire generic subset of $X$.
\end{corollary}

We note that, as strongly transitive homeomorphisms on a compact metric space $X$ are minimal, Corollary~\ref{cor:minimalFV} extends the information in Remark~\ref{rmk:minimal} to strongly transitive continuous maps on compact metric spaces.

\section{Proof of Theorem~\ref{CV_mod}}

The argument is a direct adaptation of the one in \cite[Theorem~A]{CV21}. Suppose that there exist dense subsets $A, B$ of $Y$ and $\varepsilon > 0$ such that for any $(a,b) \in A \times B$ there exist $(r_{a},r_{b}) \in \{\varphi_{n}(a):  n \geq 1 \}' \times \{\varphi_{n}(b):  n \geq 1 \}'$ satisfying $|r_{a} -r_{b}| > \varepsilon$. Fix $0 <\eta < \frac{\varepsilon}{3}$. Since the maps $\varphi_n$ are continuous, given an integer $N \in \mathbb{N}$ the set
$$\Lambda_{N} =  \Big\{y \in Y: \,|\varphi_{n}(y) - \varphi_{m}(y) | \leqslant \eta \quad \forall  m,n \geqslant N \Big\}$$
is closed in $Y$. Moreover:

\begin{lemma}\label{eq:interior.empty}
$\Lambda_{N}$ has empty interior for every $N \in \mathbb{N}$.
\end{lemma}

\begin{proof}
Assume that there exists $N \in \mathbb{N}$ such that $\Lambda_{N}$ has non-empty interior (which we abbreviate into $\intt (\Lambda_{N})\neq \emptyset$). Hence there exists $a \in A$ such that $a  \in   \intt (\Lambda_{N})$. Since $\varphi_{N}$ is continuous, there exists $\delta_{N} >0$ such that $|\varphi_{N}(a) - \varphi_{N}(y)| <  \eta$ for every $y \in Y$ satisfying $d(a,y)  < \delta_{N}$. By the denseness of $B$, one can choose $b \in B$ such that $b \in \intt (\Lambda_{N})$ and $d(a,b) < \delta_{N}$. Besides, according to the the definition of $\Lambda_{N}$ one has
$$|\varphi_{n}(a) - \varphi_{m}(a) | \leqslant \eta \quad \quad \text{ and } \quad \quad |\varphi_{n}(b) - \varphi_{m}(b)| \leqslant \eta \quad \forall, m, n \geqslant N.$$

Given $(a,b) \in  A\times B$, choose $ (r_{a},r_{b}) \in \{\varphi_{n}(a):  n \geq 1 \}' \times \{\varphi_{n}(b):  n \geq 1 \}'$ satisfying $|r_{a}-r_{b}| > \varepsilon$. Fixing $m = N$, taking the limit as $n$ goes to $+\infty$ in the first inequality along a subsequence converging to $r_{a}$ and taking the limit as $n$ tends to $+\infty$ in the second inequality along a subsequence convergent to $r_{b}$, we conclude that
$$|r_{a} - \varphi_{N}(a)| \leqslant \eta \quad \quad \text{ and }\quad \quad |r_{b} - \varphi_{N}(b)| \leqslant \eta.
$$
Therefore,
$$\varepsilon \,<\, |r_{a} - r_{b}| \,\leqslant\, |\varphi_{N}(a) - r_{a} | + |\varphi_{N}(b) - r_{b} | + |\varphi_{N}(a) - \varphi_{N}(b)| \,\leqslant\, 3 \eta$$
contradicting the choice of $\eta$. Thus $\Lambda_{N} $  must have empty interior.
\end{proof}

We can now finish the proof of Theorem \ref{CV_mod}. Using the fact that $\limsup_{n\,\to\,+\infty} \|\varphi_n\|_\infty <+\infty$, one deduces that $Y \setminus {\mathcal I}(\Phi) \subset \bigcup_{N=1}^{\infty} \Lambda_{N}$. Thus, by Lemma~\ref{eq:interior.empty}, the set of $\Phi$-regular points is contained in a countable union of closed sets with empty interior. This shows that ${\mathcal I}(\Phi)$ is Baire generic, as claimed.
The second statement in the theorem is a direct consequence of the first one. \hfill $\square$

\begin{remark}\label{remark.CV_mod}
It is worth mentioning that the proof of Theorem~\ref{CV_mod} also shows that one has $\{y \in Y: \limsup_{n} \varphi_{n}(y) - \liminf_{n} \varphi_{n}(y) < \eta \}  \subseteq \bigcup_{N=1}^{\infty} \Lambda_{N}$, and so the set
$\{y \in Y: \limsup_{n} \varphi_{n}(y) - \liminf_{n} \varphi_{n}(y) \geqslant \eta \}$ is Baire generic in $Y$.
\end{remark}

\begin{remark}\label{rmk:flows}
The argument used in the proof of Theorem~\ref{CV_mod} adapts naturally to the context of continuous-time dynamical systems. This fact will be used later, when applying Theorem~\ref{CV_mod} to geodesic flows on non-positive curvature (see Example~\ref{ex:geodesicflows}).
\end{remark}


\section{Proof of Corollary~\ref{maincorollary_auxiliar}}

We start showing that the existence of an irregular point with respect to an observable $\varphi$ whose orbit by $T$ is dense is enough to ensure that $Y$ is $(T,\varphi)$-sensitive.

\begin{lemma}\label{lemma1}
Let $(Y, d)$ be a Baire metric space,  $T : Y \to Y$ be a continuous map such that $(Y,T)$ has a dense orbit, and $\varphi \in C^b(Y,\mathbb R)$.
If ${\mathcal I}(T,\varphi) \cap \Trans (Y,T) \neq \emptyset $ then $Y$ is $(T,\varphi)$-sensitive.
\end{lemma}

\begin{proof}
Suppose that $\varphi \in C^{b}(Y,\mathbb{R})$ and ${\mathcal I}(T,\varphi) \cap \Trans (Y,T) \neq \emptyset$. Let $y \in Y$ be a point in this intersection. Then there is $\varepsilon >0$ such that $\varepsilon <  \limsup_{n \,\to\, +\infty} \varphi_{n}(y) - \liminf_{n \,\to \,+\infty} \varphi_{n}(y)$. Since $\varphi$ is a bounded function, the values
$$\liminf_{n \,\to\,+ \infty} \,\frac{1}{n} \,\sum_{j=0}^{n-1} \,\varphi(T^{j}(z)) \quad \text{ and } \quad \limsup_{n \,\to\, +\infty} \,\frac{1}{n} \,\sum\limits_{j=0}^{n-1} \,\varphi(T^{j}(z))$$
are constant for every $z \in \{T^{j}(y): j \in \mathbb{N}\cup\{0\}\}$. This invariance, combined with the fact that $\{T^{n}(y): n \in \mathbb{N}\cup\{0\}\}$ is a dense subset of $Y$, implies that $Y$ is $(T,\varphi)$-sensitive.
\end{proof}

Let us resume the proof of Corollary~\ref{maincorollary_auxiliar}.

\noindent $(i)\,$ We will adapt the argument in the proof of \cite[Lemma 3.4]{HLT21}. Take $y \in \Trans (Y,T)$. Assume that $Y$ has an isolated point and $\mathcal{D}(Y,T) \neq \emptyset$. Then there exists $N \in \mathbb{N}\cup\{0\}$ such that $T^{N}(y)$ is an isolated point of $Y$. Take $\psi \in \mathcal{D}(Y,T)$ whose set ${\mathcal I}(T,\psi)$ is dense in $Y$. Since $\{T^{N}(y)\}$ is an open subset of $Y$, one has
$\{T^{N}(y) \} \cap {\mathcal I}(T,\psi) \neq \emptyset$, so $T^{N}(y)$ belongs to ${\mathcal I}(T,\psi)$. Therefore, $T^{N}(y) \in \bigcap_{\psi \,\in\, \mathcal{D}(Y,T)} {\mathcal I}(T,\psi).$

In fact, more is true: $y \in \bigcap_{\psi\, \in \,\mathcal{D}(Y,T)} {\mathcal I}(T,\psi).$ Indeed, suppose that there exists $\psi_0 \in D(Y,T)$ such that $y \in Y \setminus {\mathcal I}(T,\psi_0)$. Consider the aforementioned integer $N \in \mathbb{N}\cup\{0\}$ such that $T^{N}(y) \in \bigcap_{\psi \,\in \,\mathcal{D}(Y,T)} {\mathcal I}(T,\psi)$. As $Y \setminus I(T,\psi)$ is $T$-invariant and $y \in Y \setminus {\mathcal I}(T,\psi_0)$, we have $T^{N}(y) \in Y \setminus {\mathcal I}(T,\psi_0)$. This contradicts the choice of $N$.
This way we have shown that $\Trans (Y,T) \subseteq \bigcap_{\psi\,\in\, \mathcal{D}(Y,T)} {\mathcal I}(T,\psi)$.

We are left to prove that, if $(Y,d)$ has an isolated point and $\mathcal{D}(Y,T) \neq \emptyset$, then $\mathcal{R}(Y,T) = \mathcal{D}(Y,T)$. Suppose that $\mathcal{D}(Y,T) \neq \emptyset$. As we have just proved, $\Trans (Y,T) \subseteq \bigcap_{\psi\,\in \,\mathcal{D}(Y,T)} {\mathcal I}(T,\psi)$. This implies that for each $\psi\in \mathcal{D}(Y,T)$ one has ${\mathcal I}(T,\psi) \cap \Trans (Y,T) \neq \emptyset$. By Lemma \ref{lemma1} and Theorem \ref{CV_mod} we conclude that ${\mathcal I}(T,\psi)$ is Baire generic, so $\psi \in \mathcal{R}(Y,T)$.

\smallskip

\noindent $(ii)\,$ Suppose that $\bigcap_{\varphi \,\in\, \mathfrak{F}} {\mathcal I}(T,\varphi)$ is Baire generic. By hypothesis, there exists $y$ in $Y$ such  that  $\Omega_{y} = \overline{\{T^{n}(y): n \in \mathbb{N}\cup\{0\}\}}  = Y$, so $Y$ is separable. This implies that, if $\omega(y) = \Omega_{y}$, then $Y$ does not have isolated points, and so $T$ is transitive. Since $Y$ is a Baire separable metric space,
$\Trans (Y,T)$ is Baire generic as well. Therefore, $\Trans (Y,T) \cap \bigcap_{\varphi \,\in\, \mathfrak{F}} {\mathcal I}(T,\varphi)$ is also Baire generic. Thus $\Trans (Y,T) \cap \bigcap_{\varphi \,\in \,\mathfrak{F}} {\mathcal I}(T,\varphi)$ is not empty.

Assume now that $\omega(y) \subsetneq  \Omega_{y}$. Then $Y$ has an isolated point. As $\emptyset \neq \mathfrak{F} \subseteq \mathcal{D}(Y,T)$ and $Y$ has an isolated point, by item $(i)$ we know that $\Trans (Y,T) \subseteq \bigcap_{\varphi\, \in \,\mathcal{D}(Y,T)}{\mathcal I}(T,\varphi)$. Moreover, as $\mathfrak{F} \subseteq \mathcal{D}(Y,T)$, one has $\bigcap_{\varphi \,\in\, \mathcal{D}(Y,T)} {\mathcal I}(T,\varphi) \subseteq \bigcap_{\varphi\, \in \,\mathfrak{F}} {\mathcal I}(T,\varphi)$. Thus, $\Trans (Y,T) \subseteq \bigcap_{\varphi \,\in \,\mathfrak{F}} {\mathcal I}(T,\varphi)$. 
\hfill $\square$


\section{Proof of Corollary~\ref{maincorollary_sieve}}

We will adapt the proof of \cite[Theorem B]{HLT21}. Consider the sets
$$A_{\ell_\varphi^{*}} = \Big\{y \in Y \colon \liminf\limits_{n\,\to\,+\infty}\,\frac{1}{n}\,\sum\limits_{j=0}^{n-1}\, \varphi (T^{j} y) = \ell_\varphi^{*} \Big\}$$
$$A_{L_\varphi^{*}} = \{y \in Y \colon \limsup\limits_{n\,\to\,+\infty}\,\frac{1}{n}\,\sum\limits_{j=0}^{n-1} \,\varphi (T^{j} y) =
L_\varphi^{*} \Big\}$$
and, for every $\alpha > \ell_\varphi^{*}$,
$$B_\alpha = \bigcap_{N = 1}^{+\infty} \bigcup_{n = N}^{+\infty}\,\Big\{y \in Y \colon \frac1n\,\sum_{i=0}^{n-1}\,\varphi(T^i y) < \alpha\Big\}.$$ As $\alpha > \ell_\varphi^{*}$, there exists $y_0\in \Trans(Y,T)$ such that $y_0 \in B_\alpha$. Consequently, for every $N \in \mathbb{N}$ the set $\bigcup_{n=N}^{+\infty}\Big\{y \in Y \colon \frac1n\,\sum_{i=0}^{n-1}\,\varphi(T^i y) < \alpha\}$ is an open, dense subset of $Y$. Therefore, $B_\alpha$ is Baire generic in $Y$ for any $\alpha > \ell_\varphi^{*}$.

Take a convergent sequence $\{\alpha_n\}_{n=1}^{+\infty}$ in $\mathbb{R}$ such that $\lim_{n\,\to\,+\infty}\alpha_n=\ell_\varphi^{*}$ and $\alpha_{n} >\ell_\varphi^{*}$ for every $n \in \mathbb{N}$. Then $\bigcap_{n=1}^{+\infty} B_{\alpha_{n}} \subseteq A_{\ell_\varphi^{*}}$.  This implies that $A_{\ell_\varphi^{*}}$ is Baire generic in $Y$. We deduce similarly that $A_{L_\varphi^{*}}$ is Baire generic in $Y$. Thus $A_{\ell_\varphi^{*}} \cap A_{L_\varphi^{*}} = I_{\varphi}[\ell_\varphi^{*},L_\varphi^{*}] $ is Baire generic in $Y$.

We are now going to show that either $I_{\varphi}[\ell_\varphi^{*},L_\varphi^{*}] \subseteq {\mathcal I}(T,\varphi)$ or $I_{\varphi}[\ell_\varphi^{*},L_\varphi^{*}] \subseteq   Y \setminus {\mathcal I}(T,\varphi)$. For any $y \in I_{\varphi}[\ell_\varphi^{*},L_\varphi^{*}]$, we have
$$\liminf\limits_{n \,\to\,+\infty}\, \frac{1}{n}\,\sum\limits_{j=0}^{n-1} \,\varphi(T^{j}(y)) =  \ell_\varphi^{*} \leqslant L_\varphi^{*}=   \limsup\limits_{n \,\to\,+\infty} \,\frac{1}{n}\,\sum\limits_{j=0}^{n-1} \,\varphi(T^{j}(y)).$$
Therefore, if $\ell_\varphi^{*} < L_\varphi^{*}$, then $I_{\varphi}[\ell_\varphi^{*},L_\varphi^{*}] \subseteq   {\mathcal I}(T,\varphi)$ and
${\mathcal I}(T,\varphi)$ is Baire generic in $Y$. Otherwise, if $\ell_\varphi^{*} = L_\varphi^{*}$, then $I_{\varphi}[\ell_\varphi^{*},L_\varphi^{*}] \subseteq   Y \setminus {\mathcal I}(T,\varphi)$. These inclusions imply that ${\mathcal I}(T,\varphi)$ is either Baire generic or  meagre, and characterize the case of a meagre ${\mathcal I}(T,\varphi)$. 
\hfill $\square$

\begin{remark} According to \cite[Proposition 3.11]{P21}, in the context of Baire ergodic maps $T$, for any Baire measurable function
$\varphi: Y \to \mathbb{R}$ there exist a Baire generic subset $\mathfrak{R}$ of $Y$ and constants $c_{-}^{\varphi}$ and $c_{+}^{\varphi}$ such that
$$\ell_\varphi(y) \,= \,c_{-}^{\varphi} \quad \quad  \text{ and } \quad \quad L_\varphi(y) \, = \, c_{+}^{\varphi} \quad \quad \forall\, y
\in \mathfrak{R}.$$
\end{remark}


\section{Proof of Corollary~\ref{maincorollary_equivalence}}

\noindent $(i) \Rightarrow (ii)$  By Remark \ref{remark.CV_mod} there exists $\eta > 0$ such that $\mathcal{D} = \{y \in Y: L_{\varphi}(y) - \ell_{\varphi}(y) \geq \eta \}$ is a Baire generic, and so dense, subset of $Y$.
Note that, for any $z$ and $w$ in $\mathcal{D}$, one has either $|L_{\varphi}(z) - \ell_{\varphi}(w)| > \frac{\eta}{2}$ or $|L_{\varphi}(w) - \ell_{\varphi}(z)| > \frac{\eta}{2}$.
This implies that $Y$ is $(T,\varphi)$-sensitive, and so, by  Theorem~\ref{CV_mod}, the set ${\mathcal I}(T,\varphi)$ is Baire generic.

\smallskip

\noindent $(ii) \Rightarrow (iii)$ Apply item ($ii$) of Corollary \ref{maincorollary_auxiliar} with $\mathfrak{F}= \{\varphi\}$.

\smallskip

\noindent $(iii) \Rightarrow (i)$ Assume that $\Trans (Y,T) \cap {\mathcal I}(T,\varphi) \neq \emptyset$. Then,
there are $y_0, \, y_1 \in \Trans (Y,T)$ (possibly equal) and convergent subsequences $(\varphi_{n_{k}}(y_0))_{k \,\in\, \mathbb{N}}$ and $(\varphi_{m_{k}}(y_1))_{k \,\in \,\mathbb{N}}$ whose limits are distinct.
Let $A$ be the (dense) orbit of $y_0$, $B$ the (dense) orbit of $y_1$ and $\varepsilon$ equal to half the distance between the limits of the convergent subsequences $(\varphi_{n_{k}}(y_0))_{k \,\in\, \mathbb{N}}$ and $(\varphi_{m_{k}}(y_1))_{k \,\in \,\mathbb{N}}$. This is the data needed to confirm that $Y$ is $(T,\varphi)$-sensitive.
 \hfill $\square$


\section{Proof of Corollary~\ref{maincorollary_compact}}

\noindent (a) Suppose that $X$ is a compact metric space that has a dense orbit. From \cite[Proposition 1 and Proposition 2]{Win10} or  \cite[Proposition 2.1 and Proposition 2.2]{HLT21}, we know that $X_{\Delta}$
is Baire generic in $X$.

We proceed by showing that $X_{\Delta} \subseteq \bigcap_{\varphi \,\in \,C(X,\mathbb{R})} \widehat{I_{\varphi}[\ell_\varphi^{*},L_\varphi^{*}]} $. Take $x$ in $X_{\Delta}$ and $\varphi \in C(X,\mathbb{R})$. We claim that $\ell_\varphi(x) \leqslant  \ell_\varphi^{*}$ and $L_\varphi^{*} \leqslant  L_\varphi(x)$. Assume, on the contrary, that $\ell_\varphi(x) > \ell_\varphi^{*}$. Then there exists a transitive point $x_0$ such that $\ell_\varphi(x) > \ell_\varphi(x_0)$. Therefore, there is a subsequence $(n_{k})_{k \,\in \,\mathbb{N}}$ such that
$$\ell_\varphi(x) > \ell_\varphi(x_0) = \lim\limits_{k \,\to\, +\infty}\, \frac{1}{n_{k}}\, \sum\limits_{j=0}^{n_{k}-1}\, \varphi (T^{j}(x_0)).$$
Reducing to a subsequence if necessary, we find $\mu \in V_{T}(x_0)$ such that $\frac{1}{n_{k}} \sum_{j=0}^{n_{k}-1} \delta_{T^{j}(y)}$ converges to $\mu$ in the weak$^*$-topology as $k$ goes to $+\infty$, and so $\ell_\varphi(x) > \ell_\varphi(x_0) = \int \varphi \,d\mu$. Since, by hypothesis, $\mu$ belongs to $V_{T}(x)$, there exists an infinite sequence $(n_{q})_{q \, \in \,\mathbb{N}}$ such that $\frac{1}{n_{q}} \sum_{j=0}^{n_{q}-1} \delta_{T^{j}(x)}$ converges to $\mu$ as $q$ goes to $+\infty$. This yields to
$$\int \varphi \,d\mu = \lim\limits_{q\, \to\,+\infty}\, \frac{1}{n_{q}}\, \sum\limits_{j=0}^{n_{q}-1} \,\varphi (T^{j}(x)) \geqslant
\ell_\varphi(x) >\ell_\varphi(x_0) = \int \varphi \,d\mu$$
which is a contradiction. Therefore, we must have $\ell_\varphi(x) \leqslant  \ell_\varphi^{*}$.

We prove similarly that $L_\varphi^{*} \leqslant  L_\varphi(x)$ for every $x \in X_{\Delta}$. Thus $x \in \widehat{I_{\varphi}[\ell_\varphi^{*},L_\varphi^{*}]}$.

\smallskip

\noindent (b) We start establishing the following auxiliary result.

\begin{lemma}\label{lemma_Baire generic_empty}
Let $(X, d)$ be a compact  metric space,  $T : X \to X$ be a continuous map with a dense orbit. If $\mathcal{R}(X,T) = \emptyset$, then
$\bigcap_{\varphi\, \in\, C(X, \mathbb{R})} \,X \setminus {\mathcal I}(T,\varphi)$ is Baire generic in $X$.
\end{lemma}

\begin{proof}
Since $\mathcal{R}(X,T) = \emptyset$, from Corollary~\ref{maincorollary_sieve} we know that ${\mathcal I}(T,\varphi)$ is meagre for every $\varphi\, \in\, C(X, \mathbb{R})$. Hence $X \setminus \mathcal{I}( T,\varphi)$ is Baire generic for all $\varphi \in C(X, \mathbb{R})$. Let $S$ be a countable, dense subset of  $C(X, \mathbb{R})$. Then $\bigcap_{\psi \,\in\, S} X \setminus \mathcal{I}(T,\psi)$ is Baire generic. We are left to prove that
$$\bigcap\limits_{\psi \,\in \,S} \,X \setminus \mathcal{I}(T,\psi) \,\subseteq\,  \bigcap_{\varphi \,\in \,C(X, \mathbb{R})} \,X \setminus {\mathcal I}(T,\varphi).$$

Take $x \in \bigcap_{\psi \,\in \,S} X \setminus {\mathcal I}(T,\psi)$ and $\varphi \in C(X, \mathbb{R})$; we need to show that $x$ belongs to $X \setminus {\mathcal I}(T,\varphi)$. As $S$ is dense, given $\varepsilon >0$ there exists $\psi \in S$ such that $\|\varphi - \psi \|_\infty < \frac{\varepsilon}{3}$. Since $x$ is in $X \setminus \mathcal{I}(T,\psi)$, there is $N \in \mathbb{N}\cup\{0\}$ such that, for every $n, m \geqslant N$, we have $\|\psi_{n}(x)- \psi_{m}(x)\| < \frac{\varepsilon}{3}$. Therefore,
$$\|\varphi_{n}(x) - \varphi_{m}(x)\| \leqslant \|\varphi_{n}(x) - \psi_{n}(x) \| + \|\psi_{n}(x) - \psi_{m}(x) \| + \|\psi_{m}(x) - \varphi_{m}(x) \| < \varepsilon$$
so $x$ is in $X \setminus {\mathcal I}(T,\varphi)$.
\end{proof}

\smallskip

Let us go back to the proof of item (b). We begin by showing that $\bigcup_{\varphi \,\in \,C(X,\mathbb{R})} \mathcal{I}( T,\varphi)$ is either Baire generic or meagre. Suppose that $\bigcup_{\varphi \,\in \,C(X,\mathbb{R})} \mathcal{I}( T,\varphi)$ is not Baire generic, so $\mathcal{R}(X,T)$ is empty. By Lemma~\ref{lemma_Baire generic_empty}, the set $\bigcap_{\varphi \,\in \,C(X, \mathbb{R})} X \setminus {\mathcal I}(T,\varphi)$ is Baire generic in $X$, and so $\bigcup_{\varphi \,\in\, C(X,\mathbb{R})} \mathcal{I}( T,\varphi)$ is a meagre set. Moreover, using compactness of $X$ and Corollary \ref{maincorollary_sieve}, there exists a Borel invariant probability measure $\mu$ such that
$$\Trans(X,T) \,\subset \,\Big\{x \in X: \lim_{n \,\to \,+\infty} \,\frac{1}{n}\, \sum_{j=0}^{n-1}\, \varphi (T^{j}(x)) \,= \,\int \varphi d\mu \quad \quad \forall\, \varphi \in C(X,\mathbb{R})\Big\}.$$
By item (a), $X_{\Delta} = \{x \in X \colon  \mu \in V_{T}(x)\}$ is Baire generic in $X$. This implies that the set
$$X_{\Delta} \cap  \bigcap_{\varphi \,\in \,C(X, \mathbb{R})} X \setminus {\mathcal I}(T,\varphi) = \Big\{x \in X: \lim_{n\, \to \,+\infty}\, \frac{1}{n} \,\sum\limits_{j=0}^{n-1} \varphi (T^{j}(x)) = \int \varphi \,d\mu \quad \forall\, \varphi \in C(X,\mathbb{R})\Big\}$$
is Baire generic in $X$.

\smallskip

\begin{proposition}\label{prop:compact}
Let $(X, d)$ be a compact metric space and $T : X \to X$ be a continuous map with a dense orbit. The following statements are equivalent:
\begin{itemize}
\item[$(i)$] $ \#  \Big(\bigcup\limits_{t\,\in\, \Trans(X,T)}\, V_{T}(t) \Big) > 1$.
\smallskip
\item[$(ii)$] $\mathcal{R}(X,T) \neq \emptyset$.
\smallskip
\item[$(iii)$] $\mathcal{R}(X,T)$ is open and dense in $C(X, \mathbb{R})$.
\smallskip
\item[$(iv)$] $ X_{\Delta} \subseteq \bigcap_{\varphi \,\in \,\mathcal{R}(X,T)} \,{\mathcal I}(T,\varphi)$.
\smallskip
\item[$(v)$] $ \bigcap_{\varphi \,\in \,\mathcal{R}(X,T)} \,{\mathcal I}(T,\varphi)$ is Baire generic.
\smallskip
\item[$(vi)$] $\Trans (X,T) \cap \bigcap_{\varphi \,\in \,\mathcal{R}(X,T)}\, {\mathcal I}(T,\varphi) \neq \emptyset$.
\smallskip
\item[$(vii)$] $ \bigcup_{\varphi \,\in \,C(X,\mathbb{R})} \,{\mathcal I}(T,\varphi)$ is Baire generic.
\end{itemize}
\end{proposition}

\begin{proof}

We will prove that $(i) \Leftrightarrow (ii)$,  $(ii) \Rightarrow \cdots \Rightarrow (vii) $ and $(vii) \Rightarrow (ii)$.

\smallskip

\noindent $(i) \Rightarrow (ii)$ Suppose that $\# \left(\cup_{x\,\in\, \Trans(X,T)} V_{T}(x)\right) > 1$. Then there exist two distinct Borel probability measures $\mu$, $\nu$ in $\cup_{x\,\in\, \Trans(X,T)} V_{T}(x)$ and $\varphi \in C(X,\mathbb R)$ such that $\int \varphi \,d\mu \neq \int \varphi \,d\nu$. This implies that $\{\mu, \nu\} \subseteq V_{T}(x)$ for all $x \in X_{\Delta}$, and so $X_{\Delta} \subseteq {\mathcal I}(T,\varphi)$. Therefore, ${\mathcal I}(T,\varphi)$ is Baire generic since, by item (a), the set $X_{\Delta}$ is Baire generic.

\smallskip

\noindent $(ii) \Rightarrow (i)$ Suppose that that there exists $\varphi \in C(X,\mathbb{R})$ such that ${\mathcal I}(T,\varphi)$ is Baire generic. By Corollary~\ref{maincorollary_equivalence}, one has $\Trans (X,T) \cap {\mathcal I}(T,\varphi) \neq \emptyset$. This implies that $\ell_\varphi(y) < L_\varphi(y)$, and so $\#V_{T}(y) > 1$.

\smallskip

\noindent $(ii) \Rightarrow (iii)$ For every $x \in X$, denote by $\mathcal{U}_{x}$ the set $\{\varphi \in C(X,\mathbb{R}): x \in {\mathcal I}(T,\varphi)\}$. If $\mathcal{U}_{x} \neq \emptyset$, then $\mathcal{U}_{x}$ is an open dense subset of $C(X,\mathbb{R})$ (see \cite[Lemma 3.2]{HLT21}) and, by Corollary~\ref{maincorollary_equivalence}, $\mathcal{R}(X,T) = \bigcup_{x \,\in\, \Trans(X,T)} \mathcal{U}_{x}$. Therefore, if $\mathcal{R}(X,T) \neq \emptyset$, then $\mathcal{R}(X,T)$ is open and dense in $C(X,\mathbb{R})$.

\smallskip

\noindent $(iii) \Rightarrow (iv)$ Suppose that $\mathcal{R}(X,T)$ is open and dense in $C(X, \mathbb{R})$. In particular, $\mathcal{R}(X,T)$ is not empty. From Corollary~\ref{maincorollary_equivalence},  $\mathcal{R}(X,T) = \{\varphi \in C(X,\mathbb{R})\colon \ell_\varphi^{*} < L_\varphi^{*}\}$. Using item $(a)$, we know that $X_{\Delta} \subseteq \bigcap_{\varphi \,\in \,C(X,\mathbb{R})} \widehat{I_{\varphi}[\ell_\varphi^{*},L_\varphi^{*}]}$.
Moreover,
$$X_{\Delta} \subseteq  \bigcap\limits_{\varphi \,\in \,C(X,\mathbb{R})} \widehat{I_{\varphi}[\ell_\varphi^{*},L_\varphi^{*}]} =  \left(\bigcap\limits_{\ell_\varphi^{*} < L_\varphi^{*}} \widehat{I_{\varphi}[\ell_\varphi^{*},L_\varphi^{*}]} \right) \cap \left(\bigcap\limits_{\ell_\varphi^{*} = L_\varphi^{*}} \widehat{I_{\varphi}[\ell_\varphi^{*},L_\varphi^{*}]}\right)$$
and
$$X_{\Delta} \subseteq  \bigcap\limits_{\ell_\varphi^{*} < L_\varphi^{*}} \widehat{I_{\varphi}[\ell_\varphi^{*},L_\varphi^{*}]}.$$
Consequently,
$$X_{\Delta} \subseteq  \bigcap\limits_{\ell_\varphi^{*} < L_\varphi^{*}} {\mathcal I}(T,\varphi) =  \bigcap\limits_{\varphi \,\in \,\mathcal{R}(X,T)} \mathcal{I}(T,\varphi).$$

\smallskip

\noindent $(iv) \Rightarrow (v)$ Suppose that $X_{\Delta} \subseteq \bigcap_{\varphi \,\in \,\mathcal{R}(X,T) } \mathcal{I}(T,\varphi)$. By item $(a)$, $X_{\Delta} $ is Baire generic; hence $\bigcap_{\varphi \,\in\, \mathcal{R}(X,T)} \mathcal{I}( T,\varphi)$ is Baire generic in $X$.

\smallskip

\noindent $(v) \Rightarrow (vi)$ Suppose that $\bigcap_{\varphi \,\in \,\mathcal{R}(X,T)} {\mathcal I}(T,\varphi)$ is Baire generic. By item $(ii)$ of Corollary~\ref{maincorollary_auxiliar}, we know that $\Trans (X,T) \cap  \bigcap_{\varphi \,\in \,\mathcal{R}(X,T)} \mathcal{I}(T,\varphi) \neq \emptyset$.

\smallskip

\noindent $(vi) \Rightarrow (vii)$ This is clear from Corollary~\ref{maincorollary_equivalence}.

\smallskip

\noindent $(vii) \Rightarrow (ii)$ Suppose that $\bigcup_{\varphi \,\in \,C(X,\mathbb{R})} \mathcal{I}( T,\varphi)$ is Baire generic. If $\mathcal{R}(X,T) = \emptyset$, using Lemma~\ref{lemma_Baire generic_empty} one deduces that    $\bigcap_{\varphi\, \in\, C(X, \mathbb{R})} X \setminus {\mathcal I}(T,\varphi)$ is Baire generic. Then
$$\emptyset \, = \, \Big(\bigcup\limits_{\varphi \,\in \,C(X,\mathbb{R})} \mathcal{I}(T,\varphi)\Big) \cap \Big(\bigcap\limits_{\varphi\, \in\, C(X, \mathbb{R})} X \setminus {\mathcal I}(T,\varphi) \Big)$$
is Baire generic as well, so it is not empty. This contradiction ensures that $\mathcal{R}(X,T) \neq \emptyset$.

\end{proof}

\noindent $(c)$ Firstly, we note that, by Corollary~\ref{maincorollary_equivalence}, one has

$\mathcal{H}(X,T) = \{\varphi \in C(X,\mathbb{R}) \colon  \,\ell_\varphi^{*} < L_\varphi^{*}\} \,\cup \,\{\varphi \in C(X,\mathbb{R}):  \,\mathcal{I}(T,\varphi) \neq \emptyset \,\text{ and } \,\ell_\varphi^{*} = L_\varphi^{*}\}$

$\mathcal{H}(X,T) = \mathcal{R}(X,T) \,\cup \,\{\varphi \in C(X,\mathbb{R}):  \,{\mathcal I}(T,\varphi) \neq \emptyset \,\text{ and } \,\ell_\varphi^{*} = L_\varphi^{*}\}.$
\smallskip

\noindent Suppose that $\mathcal{H}(X,T)$ is not empty. If $\{\varphi \in C(X,\mathbb{R}): \mathcal{I}(T,\varphi) \neq \emptyset$ and $\ell_\varphi^{*} = L_\varphi^{*} \}$ is empty, then $\mathcal{H}(X,T) = \mathcal{R}(X,T)$ is not empty. Therefore, by
Proposition~\ref{prop:compact}, we conclude that $\bigcap_{\varphi \,\in\, \mathcal{H}(X,T)} \mathcal{I}(T,\varphi)$ is Baire generic.

Assume, otherwise, that $\{\varphi \in C(X,\mathbb{R}):  {\mathcal I}(T,\varphi) \neq \emptyset$ and $\ell_\varphi^{*} = L_\varphi^{*} \}$ is not empty. Then there exists $\varphi \in C(X,\mathbb{R})$ such that $\mathcal{I}(T,\varphi) \neq \emptyset$ and $\ell_\varphi^{*} = L_\varphi^{*}$. Thus $\mathcal{I}(T,\varphi)$ is meagre by Corollary~\ref{maincorollary_sieve}. As $\bigcap_{\varphi\, \in\, \mathcal{H}(X,T)} \mathcal{I}(T,\varphi) \subseteq \mathcal{I}(T,\varphi)$, the set $\bigcap_{\varphi\, \in\, \mathcal{H}(X,T)} \mathcal{I}(T,\varphi)$ is meagre as well.

This way, we have shown that $\bigcap_{\varphi\, \in\, \mathcal{H}(X,T)} \mathcal{I}(T,\varphi)$ is either Baire generic or meagre. Actually, we have proved more: that the set $\bigcap_{\varphi\, \in\, \mathcal{H}(X,T)} \mathcal{I}(T,\varphi)$ is Baire generic if and only if $\{\varphi \in C(X,\mathbb{R}):  \mathcal{I}(T,\varphi) \neq \emptyset$ and $\ell_\varphi^{*} = L_\varphi^{*} \}$ is empty. \hfill $\square$


\section{Proof of Theorem~\ref{sensitive}}

Let $(X,d)$ be a compact metric space and $\varphi \in C(X, \mathbb R)$ such that $X$ is $(T,\varphi)$-sensitive. Recall that this means that there exist dense sets $A, B\subset X$ and $\varepsilon > 0$ such that, for any pair $(a,b) \in A \times B$ there is $(r_{a},\,r_{b}) \in \{\varphi_{n}(a):  n \in \mathbb{N}\}' \times \{\varphi_{n}(b):   n \in \mathbb{N}\}'$ satisfying $|r_{a} -r_{b}| > \varepsilon$. Using
the uniform continuity of $\varphi$, one can choose $\delta > 0$ such that $|\varphi(z) - \varphi(w)| < \frac{\varepsilon}{2}$ for every $z,w \in X$ with $d(z,w) < \delta$.

Assume, by contradiction, that ${\mathcal I}(T,\varphi)$ has empty interior and $X$ has no sensitivity to initial conditions. The latter implies that for each $\theta > 0$ there exist $x_{\theta} \in X$ and an open neighborhood $U_{x_\theta}$ of $x_{\theta}$ such that $d(T^{n}(x_{\theta}), \,T^{n}(x)) \leqslant \theta$ for every $x \in U_{x_\theta}$ and $n \in \mathbb{N} \cup \{0\}$. Choose $\theta = \frac{\delta}{3}$ and take $x_\theta$ and $U_{x_\theta}$ as above.

As ${\mathcal I}(T,\varphi)$ has empty interior, the set $X \setminus {\mathcal I}(T,\varphi)$ is dense in $X$. Therefore, there are $a \in A \cap U_{x_\theta}$, $b \in B \cap U_{x_\theta}$ and a $\varphi$-regular point $c \in U_{x_\theta}$
satisfying
\begin{center}
$|\varphi(T^{n}(a)) - \varphi(T^{n}(c))| < \frac{\varepsilon}{2}\quad$ and $\quad |\varphi(T^{m}(b)) - \varphi(T^{m}(c))| < \frac{\xi}{2} \quad \quad \forall\, n,m \in \mathbb{N} \cup \{0\}.$
\end{center}
Consequently, for every $n,m \in \mathbb{N} \cup \{0\}$, one has
\begin{center}
$\left|\frac{1}{n} \sum\limits_{j=0}^{n-1} \varphi(T^{j}(a))  -   \frac{1}{n} \sum\limits_{j=0}^{n-1}\varphi(T^{j}(c)) \right| < \frac{\varepsilon}{2}\quad $ and $\quad \left|\frac{1}{m} \sum\limits_{j=0}^{m-1} \varphi(T^{j}(b)) -  \frac{1}{m} \sum\limits_{j=0}^{m-1} \varphi(T^{j}(c)) \right|  < \frac{\varepsilon}{2}.$
\end{center}
Taking the limit as $n$ goes to $+\infty$ in the first inequality along a subsequence $(n_k)_{k \, \in \, \mathbb{N}}$ converging to $r_{a}$ and taking the limit as $m$ tends to $+\infty$ in the second inequality along a subsequence $(m_k)_{k \, \in \, \mathbb{N}}$ convergent to $r_{b}$, and using that $c$ is a $\varphi$-regular point, we obtain
$$\left|r_{a} - \lim\limits_{k \,\to\, +\infty}\,\frac{1}{n_k} \,\sum\limits_{j=0}^{n_k-1}\,\varphi(T^{j}(c)) \right| \leqslant \frac{\varepsilon}{2} \quad\text{and}\quad
\left|r_{b} - \lim\limits_{k \,\to \,\infty}\,\frac{1}{m_k} \,\sum\limits_{j=0}^{m_k-1}\,\varphi(T^{j}(c)) \right| \leqslant \frac{\varepsilon}{2}$$
and so $\left|r_{a} - r_{b} \right| \leqslant \varepsilon$.
We have reached a contradiction, thus proving that either $X$ has sensitivity on initial conditions or the irregular set ${\mathcal I}(T,\varphi)$ has non-empty interior, as claimed.

\smallskip

Item (a) is a direct consequence of the previous statement and the fact that periodic points of $T$ are $\varphi$-regular. Item (b) follows from item (a). 
\hfill $\square$


\section{Proof of Corollary~\ref{cor:minimalFV}}

Let $T: X\to X$ be a strongly transitive continuous endomorphism of a compact metric space $X$. Given $\varphi\in C(X,\mathbb R)$ satisfying
$\mathcal I(T,\varphi)\neq \emptyset$, let us show that $\mathcal I(T,\varphi)$ is a Baire generic subset of $X$. Fix $x_0 \in \mathcal I(T,\varphi)$ and let $\vep:= L_\varphi(x_0)-\ell_\varphi(x_0)>0.$ The strong transitivity assumption ensures that, for every non-empty open
subset $U$ of $X$, there is $N \in \mathbb{N} \cup \{0\}$ such that $x_0 \in T^N(U)$. Thus, the pre-orbit $\mathcal O_T^-(x_0):= \,\big\{x \in X \colon \,T^n(x)=x_0\, \text{ for some $n \in \mathbb{N}$} \big\}$
of $x_0$ is dense in $X$. Moreover,
$$\sup_{r,\,s\,\in \,W_{\varphi}(x)} \,\,|r-s| > \varepsilon\, \qquad \forall\, x \in \mathcal O_T^-(x_0).$$
This proves that $X$ is $(T,\varphi)$-sensitive and so, by Theorem~\ref{CV_mod}, the set $\mathcal I(T,\varphi)$ is Baire generic in $X$.
\hfill $\square$


\section{Examples}\label{sec:examples}

In this section we discuss the hypothesis, and derive some consequences, of the main results. The first example illustrates the Definition~\ref{defTphisensitive}.

\begin{example}\label{ex:exemplo1}

Consider the shift space $X = \{0,1\}^\mathbb{N}$ endowed with the metric defined by
$$d((a_n)_{n\, \in \,\mathbb{N}}, \,(b_n)_{n\, \in \,\mathbb{N}}) \,=\, \sum_{n=1}^{+\infty} \,\frac{|a_n - b_n|}{2^n}$$
and take the shift map $\sigma: X \to X$ given by $\sigma(\, (a_n)_{n \, \in \mathbb{N}}\,)=  (a_{n+1})_{n \, \in \mathbb{N}}$. Then the sets
\begin{eqnarray*}
A &=& \big\{(a_n)_{n\, \in \,\mathbb{N}} \,|\, \,\exists\, N \in \,\mathbb{N} \colon \,a_n = 0 \quad \forall\, n \geqslant N\big\} \\
B &=& \big\{(a_n)_{n\, \in \,\mathbb{N}} \,|\, \,\exists\, N \in \,\mathbb{N} \colon \,a_n = 1 \quad \forall\, n \geqslant N\big\}
\end{eqnarray*}
are the stable sets of the fixed points $\overline{0}$ and $\overline{1}$, and are dense subsets of $X$. Besides, if $\varphi \in C(X, \mathbb{R})$, then for every $a \in A$ and $b \in B$ one has
$$\lim_{n \, \to \, +\infty}\, \varphi(\sigma^n(a)) = \varphi(\overline{0}) \quad \quad \text{and} \quad \quad \lim_{n \, \to \, +\infty}\, \varphi(\sigma^n(a)) = \varphi(\overline{1}).$$
Therefore,
$$\lim_{n \, \to \, +\infty}\, \frac{1}{n}\,\sum_{j=1}^{n-1}\,\varphi(\sigma^j(a)) = \varphi(\overline{0}) \quad \quad \text{and} \quad \quad \lim_{n \, \to \, +\infty}\, \frac{1}{n}\,\sum_{j=1}^{n-1}\,\varphi(\sigma^j(b)) = \varphi(\overline{1}).$$
And so, if $\varphi(\overline{0}) \neq \varphi(\overline{1})$ and one chooses $\varepsilon = |\varphi(\overline{0}) - \varphi(\overline{1})|/2$, then we conclude that for every $(a,b) \in A \times B$ there are $r_a \in W_\varphi(a)$ and $r_b \in W_\varphi(b)$ such that $|r_a - r_b| > \varepsilon$. Consequently, $X$ is $(\sigma, \varphi)$-sensitive with respect to any $\varphi \in C(X, \mathbb{R})$ satisfying $\varphi(\overline{0}) \neq \varphi(\overline{1})$. So,  by Theorem~\ref{CV_mod}, for every such maps $\varphi$ the set ${\mathcal I}(T,\varphi)$ is Baire generic in $X$.

\begin{remark}
We note that a similar reasoning shows that, if $T: Y \to Y$ is a continuous map acting on a Baire metric space such that $T$ has two periodic points with dense pre-orbits, then there exists $\varphi \in C^b(Y, \mathbb{R})$ whose set ${\mathcal I}(T,\varphi)$ is Baire generic in $Y$.
\end{remark}

\end{example}

The second example helps to clarify the requirements in Theorem~\ref{CV_mod} and Corollary~\ref{maincorollary_auxiliar}, besides calling our attention to the difference between transitivity and the existence of a dense orbit.

\begin{example}\label{ex:exemplo2}
Consider the space $X= \{1/n : n\in \mathbb N\} \cup \{0\}$ endowed with the Euclidean metric. Let $T: X \to X$ be the continuous map given by $T(0)=0$ and $T(1/n) = 1/(n+1)$ for every $n \in \mathbb{N}$. Notice that $X$ has isolated points and $T$ has a dense orbit, though
$\Trans(X,T)=\{1\}$. However, $T$ is not transitive. Besides, ${\mathcal I}(T,\varphi) = \emptyset$ for every $\varphi \in C(X, \mathbb R)$.
\end{example}

The next example shows that the existence of a discontinuous first integral $L_\varphi$ with two dense level sets for a continuous map $T: Y \to Y$ acting on a metric space $Y$ may be indeed stronger than requiring $(T,\varphi)$-sensitivity.

\begin{example}\label{ex:Bowen}
Let $(\Psi_t)_{t\,\in\,\mathbb R}$ be a smooth Morse-Smale flow on $\mathbb S^2$ with hyperbolic singularities $\{\sigma_1, \sigma_2, \sigma_3, \sigma_4, S, N\}$ and displaying an attracting union of four separatrices, as illustrated in Figure~\ref{figure1}. More precisely,  there exist  four separatrices $\gamma_{1}, \gamma_{2}, \gamma_{3}$ and $\gamma_4$ associated to hyperbolic singularities $\sigma_1, \sigma_2$ of saddle type, while all the other singularities are repellers.
\begin{figure}[!htb]
\begin{center}
\includegraphics[scale=0.27]{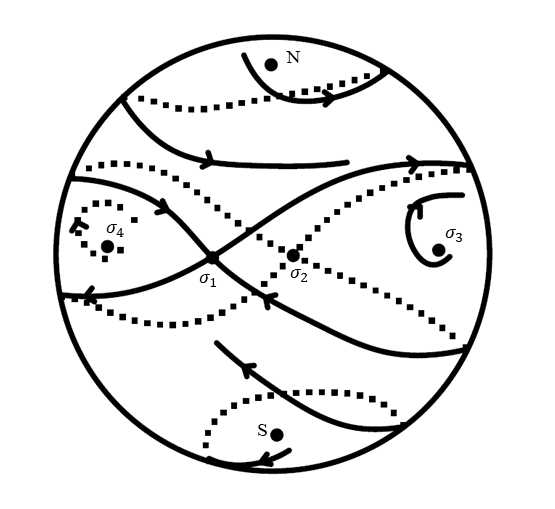}
\caption{Flow with historic behavior.}
\label{figure1}
\end{center}
\end{figure}
Let $\varphi: \mathbb S^2 \to \mathbb R$ be a continuous observable satisfying $\varphi(x) \in [0,1]$ for every $x \in \mathbb S^2$,
$\varphi(\sigma_1)= 1$ and $\varphi(\sigma_2)=0$. Consider the time-one map $T: Y \to Y$ of the flow $(\Psi_t)_{t\,\in\,\mathbb R}$. As the orbits by $T$ of all the points in $Y := \,\mathbb S^2\setminus (\bigcup_{i=1}^4 \gamma_i)$ accumulate on the closure of the union of the separatrices, it is immediate that the first integral $L_\varphi$, defined by ~\eqref{eq:firstintegral}, is everywhere constant in $Y$: $L_\varphi(y)=1$ for every $y \in Y$. However, as $W_\varphi(y) = [0,1]$ for every $y \in Y$, the space $Y$ is $(T,\varphi)$-sensitive.
\end{example}

In the following example we will address the irregular set in the context of countable Markov shifts. These symbolic systems appear naturally as models for non-uniformly hyperbolic dynamical systems on compact manifolds (see \cite{P11} and references therein), hyperbolic systems with singularities \cite{CWZ}, including Sinai dispersing billiards, and certain classes of piecewise monotone interval maps \cite{Hofbauer}, which encompass the piecewise expanding Lorenz interval maps, just to mention a few.

\begin{example}\label{ex:CMS}
Let $\mathcal{A}$ be a countable set, $\mathbb A = (a_{i,j})_{i,j\,\in \,\mathcal{A}}$ be a matrix of zeroes and ones and $\Sigma_{\mathbb A} \subset \mathcal{A}^{\mathbb N}$ be the subset
$$\Sigma_{\mathbb A} = \big\{(x_n)_{n \, \in \mathbb{N}} \in \mathcal{A}^{\mathbb N} \colon \,a_{{x_n},{x_{n+1}}} = 1  \quad \forall\, n \in \mathbb{N}\big\}.$$
Endow $\Sigma_{\mathbb A}$ with the metric
$$d((x_n)_{n \, \in \mathbb{N}},\,(y_n)_{n \, \in \mathbb{N}}) = \left\{
\begin{array}{ll}
2^{-\min\,\{k \, \in \,\mathbb{N} \colon \,x_k \,\neq\, y_k\}} & \text{ if $\{k \, \in \mathbb{N} \colon \,x_k \neq y_k\} \neq \emptyset$}\\
0 & \text{ otherwise}.
\end{array}
\right.$$
We note that $a_{i,j}=0$ for all but finitely many values of $(i,j) \in \mathcal{A}\times \mathcal{A}$ if and only if $\Sigma_{\mathbb A}$ is a compact metric space. Besides, the metric space $(\Sigma_{\mathbb A},d)$ has a countable basis of the topology, generated by the countably many cylinders, and it is invariant by the shift map $\sigma \colon \mathcal{A}^{\mathbb N} \to \mathcal{A}^{\mathbb N}$.

Observe also that the set of periodic points of $\sigma_{|\Sigma_{\mathbb A}}$ is dense in $\Sigma_{\mathbb A}$. Thus, if $\sigma_{|\Sigma_{\mathbb A}}$ is transitive, then all points have dense pre-orbits and there exists $\varphi \in C^b(\Sigma_{\mathbb A}, \mathbb R)$ such that $\Sigma_{\mathbb A}$ is $(\sigma,\varphi)$-sensitive. In fact, for each $\varphi \in C^b(\Sigma_{\mathbb A},\mathbb R)$ either $\Sigma_{\mathbb A}$ is $(\sigma,\varphi)$-sensitive or $\varphi$ complies with the following rigid condition: there exists $c_\varphi \in \mathbb R$ satisfying
$$\frac1{\pi(p)}\,\sum_{j=0}^{\pi(p)-1} \,\varphi(\sigma^j(p)) = c_\varphi \quad\text{for every periodic point $p$}$$
where $\pi(p)\in \mathbb{N}$ denotes the minimal period of $p$.

\end{example}

Our next example concerns a skew-product admitting two invariant probability measures whose basins of attraction have positive Lebesgue measure and are dense.

\begin{example}\label{ex:Kan}
Consider the annulus $\mathbb A=\mathbb S^1\times [0,1]$ and the map $T: \mathbb A \to \mathbb A$ given by
$$
T(x,t) \,= \, \Big(\, 3x \,(\mbox{mod}~1),\, t+ \frac{t(1-t)}{32} \cos (2\pi x)  \Big) \quad \quad \forall \, (x,t) \in \mathbb S^1 \times [0,1].
$$
In \cite{Kan}, Kan proved that $T$ admits two physical measures, namely $\mu_0 = \mbox{Leb}_{\mathbb S^1} \times \delta_0$ and $\mu_1 = \mbox{Leb}_{\mathbb S^1} \times \delta_1$, 
whose basins of attraction are intermingled, that is, for every non-empty open set $\mathcal U \subset \mathbb A$
$$
\mbox{Leb}_{\mathbb A} \big( \mathcal U \cap B(\mu_1) \big) > 0  \quad \text{ and } \quad \mbox{Leb}_{\mathbb A}\big( \mathcal U \cap B(\mu_2) \big) > 0.
$$
Later, Bonatti, D\'iaz and Viana introduced in \cite{BDV} the concept of Kan-like map 
and proved that any such map robustly admits two physical measures. More recently, Gan and Shi \cite{GS} showed that, in the space of $C^2$ diffeomorphisms of $\mathbb A$ preserving the boundary, every $C^2$ Kan-like map $T_0$ admits a $C^2$-open neighborhood $\mathscr V$ such that the following holds: for each $T \in \mathscr V$ and every non-empty open set $\mathcal U \subset \mathbb A$,
$$\text{the interior of }(\mathbb A) \,\subset \bigcup_{n \, \geqslant \,0} T^n(\mathcal U).$$

Using Theorem~\ref{CV_mod} we conclude that, if $T$ is a Kan-like map and $\varphi \in C(\mathbb A, \mathbb R)$ satisfies $\int \varphi \,d\mu_0 \neq \int \varphi \,d\mu_1$, then $\mathcal I(T,\varphi)$ is a Baire generic subset of $\mathbb A$. More generally, 
the argument that established Corollary~\ref{cor:minimalFV} also ensures that, for any $\varphi \in C(\mathbb A, \mathbb R)$, one has
\begin{itemize}
\item[(a)] either $\mathcal I(T,\varphi) \cap \text{interior }(\mathbb A)=\emptyset$
\smallskip
\item[(b)] or $\mathcal I(T,\varphi)$ is a Baire generic subset of $\mathbb A$.
\end{itemize}
\end{example}

\section{Applications}\label{sec:applications}

As it will become clear in the remainder of this section, Theorem~\ref{CV_mod} has a wide range of applications according to the class of sequences $\Phi=(\varphi_n)_{n\ge 1}$ of observables one considers. Let us provide two such applications, one with a geometric motivation and another in the context of semigroup actions.

\noindent \textbf{Application 1}. The irregular sets of uniformly hyperbolic maps and flows on compact Riemannian manifolds have been extensively studied. One of the reasons for this success is that these dynamical systems can be modeled by symbolic dynamical systems which satisfy the so-called specification property. Irregular sets for continuous maps acting on compact metric spaces and satisfying the specification property have been studied in \cite{LW14}. Many difficulties arise, though, if one drops the compactness assumption. An important example of a hyperbolic dynamical system with non-compact phase space is given by the geodesic flow on a complete connected negatively curved manifold. The next example applies Theorem~\ref{CV_mod} precisely to this setting.

\smallskip

\begin{example}\label{ex:geodesicflows} Let $(M,g)$ be a connected, complete Riemannian manifold. We will discuss the Baire genericity of points with historic behavior in the following two cases:
\begin{itemize}
\item[(I)] $(M,g)$ is negatively curved and the non-wandering set of the geodesic flow $(\Psi^g_t)_{t\,\in\,\mathbb R}$ contains more than two periodic orbits.
\smallskip
\item[(II)] $(M,g)$ has non-positive curvature, its universal curvature has no flat strips and the geodesic flow $(\Psi^g_t)_{t\,\in\,\mathbb R}$ has at least three periodic orbits.
\end{itemize}

Regarding (I), by \cite[Theorem~1.1]{CS10} it is known that the space $\mathcal E$ of Borel ergodic probability measures fully supported on the non-wandering set $\Omega$ is a $G_\delta$-dense subset of all Borel probability measures on $T^1 M$ which are invariant by the geodesic flow. In particular, one can choose a continuous observable $\varphi: T^1 M \to \mathbb R$ such that
$$\inf_{\mu\,\in \,\mathcal E} \,\int \varphi\, d\mu < \sup_{\mu\,\in \,\mathcal E} \,\int \varphi\, d\mu.$$
As the ergodic basins of attraction of the probability measures in $\mathcal E$ are dense in $\Omega$, one concludes that
$$L_\varphi(\cdot) := \,\limsup_{t\,\to\, +\infty} \,\frac1t \,\int_0^t \varphi(\Psi^g_s(\cdot))\, ds$$
is a first integral for the geodesic flow $(\Psi^g_t)_{t\,\in\,\mathbb R}$. Moreover, there are subsets $A, B \subset T^1 M$ which are dense in $\Omega$ and whose $L_\varphi$ value is constant, though the value in $A$ is different from the one in $B$. The existence of $A$ and $B$ means that $\Omega$ is $(T, \Phi)$-sensitive, where $T = \Psi^g_1$ is the time-1 map of the geodesic flow and $\Phi$ is defined by $\Phi = \int_0^1 (\varphi \circ \Psi^g_s) \, ds$. Then Theorem~\ref{CV_mod} ensures that $\mathcal{I}(T,\Phi)$ is Baire generic. Consequently,
$$\mathcal{I}((\Psi^g_t)_{t\,\in\,\mathbb R},\,\varphi) :=  \Big\{y \in \Omega : \, \lim\limits_{t \,\to\,+ \infty} \,\frac1t \,\int_0^t \varphi(\Psi^g_s(y))\, ds \quad\text{does not exist}\Big\}$$
is a Baire generic subset of $\Omega$ as well.

\smallskip

The previous argument can be repeated in the case (II) to yield the conclusion that $I((\Psi^g_t)_{t\,\in\,\mathbb R},\,\varphi)$ is Baire generic since, according to \cite[Theorem~1.1]{CS14}, the space $\mathcal E$ of Borel ergodic probability measures with full support on the non-wandering set $\Omega$ also form a $G_\delta$-dense set.
\end{example}

\smallskip

\noindent \textbf{Application 2}.\label{ex:amenable} Recall that a locally compact group $G$ is \emph{amenable} if for every compact set
$K \subset G$ and $\delta > 0$ there is a compact set $F\subset G$, called $(K,\delta)$-invariant, such that $m(F \Delta K F) < \delta \,m(F)$, where $m$ denotes the counting measure on $G$ if $G$ is discrete, and stands for the Haar measure in $G$ otherwise. We refer the reader to \cite{OW87} for alternative formulations of this concept. A sequence $(F_n)_n$ of compact subsets of $G$ is a \emph{F{\o}lner sequence} if, for every compact $K \subset G$, the set $F_n$ is $(K, \delta)$-invariant for every sufficiently large $n \in \mathbb{N}$ (whose estimate depends on $K$). A F{\o}lner sequence $(F_n)_n$ is \emph{tempered} if there exists $C > 0$ such that
$$m\Big(\bigcup_{1\, \leqslant \, k \,< \,n} \,F_k^{-1} \,F_n \,\Big) \leqslant C \,m\big(\,F_n\,\big) \qquad \forall\, n \in \mathbb{N}.$$
It is known that every F{\o}lner sequence has a tempered subsequence and that every amenable group has a tempered F{\o}lner sequence (cf. \cite[Proposition 1.4]{Lin}). Furthermore, if $G$ is an amenable group acting on a probability space $(X,\mu)$ by measure preserving maps and $(F_n)_n$ is a tempered F{\o}lner sequence, then for every $\varphi\in L^1(\mu)$ the limit
$$\lim_{n\,\to\,+\infty} \,\frac1{m(F_n)} \,\int_{F_n} \varphi(g(x)) \, dm(g)$$
exists for $\mu$-almost every $x \in X$; if, in addition, the $G$-action is ergodic, the previous limit is $\mu$-almost everywhere constant and coincides with $\int \varphi\, d\mu$ (cf. \cite[Theorem~1.2]{Lin}).

\smallskip

Let $(X,d)$ be a compact metric space, $G$ be an amenable group. 
We say that a Borel probability measure $\mu$ on $X$ is \emph{$G$-invariant} (or invariant by the action $\Gamma: G \times X \to X$ of $G$ on $X$) if $\mu({\Gamma_g}^{-1}(A)) = \mu(A)$ for every measurable set $A$ and every $g \in G$. We denote the space of $G$-invariant probability measures by $\mathcal M_G(X)$. A group action of $G$ on $X$ is said to be \emph{uniquely ergodic} if it admits a unique $G$-invariant probability measure.
\smallskip

Consider $\varphi \in C(X,\mathbb R)$ and the sequence $\Phi=(\varphi_n)_{n \, \in \, \mathbb{N}}$ of continuous and bounded maps defined by
\begin{equation}\label{eq:Folner-sequence}
\varphi_n(x) = \,\frac1{m(F_n)} \,\int_{F_n} \,\varphi(g(x)) \, dm(g).
\end{equation}
Clearly, $\|\varphi_n\|_\infty \leqslant \|\varphi\|_\infty$ for every $n \in \mathbb{N}$. Besides, if there are fully supported $G$-invariant and ergodic Borel probability measures $\mu_1 \neq \mu_2$, then $X$ is $\Phi$-sensitive. Thus, Theorem~\ref{CV_mod} ensures that
$$\Big\{x \in X  \colon \,\lim_{n\,\to\,+\infty} \,\frac1{m(F_n)} \,\int_{F_n} \varphi(g(x)) \, dm(g) \, \text{ does not exist} \Big\}$$
is a Baire generic subset of $X$. In what follows, we will show that the previous requirement is satisfied by countable amenable group actions with the specification property.

\begin{definition}\label{def:specification-G}
Following \cite{CL,RTZ}, we say that the continuous
group action $\Gamma: G \times X \to X$ has the \emph{specification property} if for every $\vep>0$ there exists a finite set $K_\vep\subset G$ (depending on $\vep$) such that the following holds: for any finite sample of points $x_0, x_1, x_2, \dots, x_\kappa$ in $X$ and any collection of finite subsets $\hat F_0, \hat F_1, \hat  F_2, \dots, \hat  F_\kappa$ of $G$ satisfying the condition
\begin{equation}\label{eq:spec1}
K_\vep \hat  F_i \cap \hat F_j = \emptyset \qquad \text{for every distinct\; } 0\leqslant i,j \leqslant \kappa
\end{equation}
there exists a point $x\in X$ such that
\begin{equation}\label{eq:spec2}
d(\Gamma_g(x), \Gamma_g(x_i)) < \vep \qquad \text{for every\; } g\in \bigcup_{0 \,\leqslant\, j \,\leqslant \,\kappa} \hat F_j.
\end{equation}
\end{definition}

In rough terms, the previous property asserts that any finite collection of pieces of orbits can be shadowed by a true orbit provided that there is no overlap of the (translated) group elements that parameterize the orbits. We note that, if $G$ is generated by a single map $g$, then Definition~\ref{def:specification-G} coincides with the classical notion of specification for $g$.

As a first consequence of this property one can show that the \emph{basin of attraction} of a $G$-invariant probability measure $\mu$, defined by
$$B(\mu) = \Big\{x \in X \colon \,\frac1{|F_n|}\, \sum_{g\,\in \,F_n} \,\delta_{g(x)} \to \mu \; \; \text{(in the weak$^*$ topology)} \Big\}$$
is topologically large.

\begin{lemma}\label{le:basin}
Let $G$ be a countable amenable group, $(F_n)_{n\, \in \, \mathbb{N}}$ be a tempered F{\o}lner sequence, $(X,d)$ be a compact metric space and
$\Gamma : G \times X \to X$ be a continuous group action satisfying the specification property. If $\mu$ is a $G$-invariant and ergodic probability measure on $X$ then $B(\mu)$ is a dense subset of $X$.
\end{lemma}

\begin{proof}
As $X$ is compact, the space $C(X,\mathbb R)$ is separable. Given a dense sequence $(\varphi_\ell)_{\ell\, \in \,\mathbb{N}}$ in $C(X,\mathbb R)$ and $\ell \in \mathbb{N}$, by the ergodic theorem for amenable group actions \cite{Lin} there is a full $\mu$-measure subset $X_{\ell}\subset X$ such that, for every $x \in X_\ell$,
$$\lim_{n\,\to\, +\infty} \,\frac1{|F_n|} \,\sum_{g\,\in \,F_n} \,\varphi(g(x)) = \int \varphi\, d\mu.$$
Therefore, $\bigcap_{\ell\, \in \,\mathbb{N}} X_\ell \subset B(\mu)$ and $\mu(B(\mu)) = 1$. Take $z\in B(\mu)$ and observe that, as $\lim_{n\to\infty} |F_n|=+\infty$, for each $\varphi\in C(X,\mathbb R)$ and any compact set $F\subset G$ one has
$$\lim_{n\,\to\, +\infty} \; \frac1{|F_n \setminus F|} \,\sum_{g\,\in \,F_n\setminus F} \,\varphi(g(z)) = \int \varphi\, d\mu.$$
Now, for each $\vep>0$, let $K_\vep\subset G$ be given by the specification property. Consider the set $K_\vep^{-1}:=\{g^{-1} \colon g\in K_\vep\}$ and notice that, by the compactness of $X$, the modulus of uniform continuity $\zeta_\vep:= \sup\{ |\varphi(u)-\varphi(v)| \colon \, v\in B(u,\vep), \; u\in X  \}$, where $B(u,\vep)$ stands for the ball in $X$ centered in $u$ with radius $\varepsilon$, tends to zero as $\vep$ goes to $0$.

Fix an arbitrary point $x\in X$ and $\delta>0$. For each $k \in \mathbb{N}$ one can choose a positive integer $n_k \gg 1$ such that, as $k$ goes to $+\infty$,
$$\frac{K_{\delta/2^k}}{|F_{n_k}|} \to 0 \quad \quad  \text{and} \quad \quad \frac{|F_{n_{k-1}}|}{|F_{n_k}|} \to 0.$$
We claim that $B(\mu)$ intersects the closed ball $\overline{B(x,\delta)}$ of radius $\delta$ around $x$. The idea to prove this assertion is to shadow pieces of orbits of increasing size in the basin of $\mu$ which start in the ball $B(x,\delta)$. More precisely, taking $x_0=x$, $x_1=z$ and the finite sets $\hat F_0=\{id\}$ and $\hat F_1 = (K_{\delta/2}^{-1} [F_{n_1} \setminus \hat F_0]) \setminus K_{\delta/2}$, which satisfy ~\eqref{eq:spec1}, by the specification property there is $y_1 \in B(x,\delta/2)$ such that $d(\Gamma_g(y_1), \Gamma_g(z)) < \vep$ for every $g\in \hat F_1.$
In particular, 
$$\Big| \frac1{|\hat F_1|} \sum_{g\,\in \,\hat F_1} \varphi({g(y_1)}) - \frac1{|\hat F_1|} \sum_{g\,\in \,\hat F_1} \varphi({g(z)}) \Big| < \zeta_{\delta/2}.$$
Now set $x_2 = y_1$, $x_3=z$, $\hat F_2:= \hat F_0 \cup  \hat F_1$ and $\hat F_3 = (K_{\delta/2^2}^{-1} [F_{n_2} \setminus \hat F_2]) \setminus (K_{\delta/2^2} \hat F_2)$. Hence, if $\hat F_4:= \hat F_2 \cup  \hat F_3$, using the specification property once more one obtains a point $y_2 \in B(y_1, \delta/2^2)$ such that $d(\Gamma_g(y_2), \Gamma_g(z)) < \vep$ for every $g\in \hat F_3$, from which it is immediate that
$$\Big| \frac1{|\hat F_3|} \sum_{g\,\in \,\hat F_3} \varphi({g(y_1)}) - \frac1{|\hat F_3|} \sum_{g\,\in \,\hat F_3} \varphi({g(z)}) \Big| < \zeta_{\delta/2^2}.$$
Proceeding recursively, given $y_j \in B(y_{j-1},\delta/2^{j})$ and a finite set $\hat F_j\subset G$ containing $\{id\}$ we take
\begin{equation}\label{eq.disjoint-tile}
\hat F_{j+1}= (K_{\delta/2^j}^{-1} [F_{n_{j+1}} \setminus \hat F_j]) \setminus (K_{\delta/2^j} \hat F_j)
\end{equation}
and, by the specification property, we find $y_{j+1}\in B(y_{j},\delta/2^{j+1})$ satisfying
$$d(\Gamma_g(y_{j+1}), \Gamma_g(y_j)) < \vep \quad\text{for every } g\in \hat F_{j}$$
and
$$d(\Gamma_g(y_{j+1}), \Gamma_g(z)) < \vep \quad\text{for every } g\in \hat F_{j+1}.$$
Thus, by construction, $(y_k)_{k \in \mathbb{N}}$ is a Cauchy sequence in $\overline{B(x,\delta)}$, hence convergent to some point $y_\infty \in \overline{B(x,\delta)}$. Moreover, the choice of the sets $\hat F_k$ implies that
$$\lim_{k\,\to\, +\infty} \,\frac{|\hat F_k \, \Delta\, F_{n_k}|}{|F_{n_k}|} = 0$$
and
$$\Big| \frac1{|\hat F_{j+1}|} \sum_{g\,\in\, \hat F_{j+1}} \varphi({g(y_{j+1})}) - \frac1{|\hat F_{j+1}|} \sum_{g\,\in\, \hat F_{j+1}} \varphi({g(z)}) \Big| < \zeta_{\delta/2^j}$$

\noindent from which we conclude that $y_\infty$ belongs to the basin of $\mu$. 
\end{proof}

The following dichotomy is a direct consequence of Lemma~\ref{le:basin} and Theorem~\ref{CV_mod}.

\begin{corollary}\label{cor:minimal1}
Let $G$ be a countable amenable group, $(F_n)_{n\, \in \,\mathbb{N}}$ be a tempered F{\o}lner sequence, $X$ be a compact metric space and
$\Gamma : G\times X \to X$ be a continuous group action satisfying the specification property. Then for every $\varphi \in C(X, \mathbb{R})$ either
$$\int \varphi \, d\mu_1 = \int \varphi \, d\mu_2 \quad \quad \forall\, \mu_1, \, \mu_2 \in \mathcal M_G(X)$$
or
$$\Big\{x \in X  \colon \,\lim_{n\,\to\,+\infty} \,\frac1{|F_n|} \,\sum_{g\,\in \,F_n} \,\varphi(g(x))\, \text{ does not exist} \Big\}$$
is a Baire generic subset of $X$.

\end{corollary}

We observe that Corollary~\ref{cor:minimal1} has an immediate consequence regarding the empirical measures distributed along elements of a tempered F{\o}lner sequence: under the assumptions of this corollary, one has that either the amenable group action is uniquely ergodic or the set 
$$\big\{x \in X \colon \,\Big(\frac1{|F_n|} \,\sum_{g\,\in \,F_n} \,\delta_{g(x)}\Big)_{n\, \in \, \mathbb{N}} \, \text{ does not converge in the weak$^*$ topology} \big\}$$
is Baire generic in $X$. This extends Furstenberg's theorem (see \cite[Theorem 3.2.7]{Ollagnier}), according to which an amenable group action by homeomorphisms on a compact metric space is uniquely ergodic if and only if the averages \eqref{eq:Folner-sequence} of every continuous function converge to a constant.

\smallskip

\begin{example}\label{ex:Anosov-abelian}

Consider the $2$-torus $\mathbb{T}^2$, with a Riemannian metric $d$, and the linear Anosov diffeomorphisms $g_1$ and $g_2$ of $\mathbb{T}^2$ induced by the matrices
$A_1 = \left(\begin{array}{cc}
2 & 1 \\
1 & 1 \\
\end{array}
\right)$ and
$A_2 = \left(\begin{array}{cc}
1 & 1 \\
1 & 0 \\
\end{array}
\right)$.
As $A_1 = A_2^2$, the maps $g_1$ and $g_2$ commute and induce a $\mathbb Z^2$-action on the 2-torus 
given by
\begin{eqnarray*}
\Gamma: \mathbb Z^2 \times \mathbb T^2 \quad &\to& \quad \mathbb T^2 \\
\big((m,n),\, x\big) \quad &\mapsto& \quad \big(g_1^m \circ g_2^n\big)(x).
\end{eqnarray*}
Moreover, this action satisfies the specification property. Let us see why.

\smallskip

Given $\vep>0$, let $k_\vep \in \mathbb{N}$ be provided by the specification property for the Anosov diffeomorphism $g_2$ and $K_\vep = [-k_\vep,\,k_\vep]^2 \subset \mathbb Z^2$. For any finite collection of points $x_0, x_1, x_2, \dots, x_\kappa$ in $\mathbb{T}^2$ and any choice of finite subsets $\hat F_0, \hat F_1, \hat  F_2, \dots, \hat  F_\kappa$ of $\mathbb Z^2$ satisfying the condition
\begin{equation}\label{eq:spec1Z2}
K_\vep \hat  F_i \cap \hat F_j = \emptyset \qquad \text{for every distinct\; } 0\leqslant i,j \leqslant \kappa
\end{equation}
we claim that there exists a point $x \in \mathbb{T}^2$ such that
\begin{equation*}
d(\Gamma_{(m,n)}(x), \Gamma_{(m,n)}(x_i)) < \vep \qquad \text{for every\; } (m,n) \in \bigcup_{0 \,\leqslant\, j \,\leqslant \,\kappa} \hat F_j.
\end{equation*}
Indeed, consider the map 
$\Theta: \mathbb Z^2 \to \mathbb Z$ given by $\Theta((m,n))=2m+n$ and notice that
\begin{equation}\label{eqrelTheta}
\Gamma((m,n), \cdot)= g_2^{\Theta((m,n))}(\cdot) \quad \quad \forall\, m,n\in \mathbb Z.
\end{equation}
The choice of the sets $K_\vep$ together with assumption \eqref{eq:spec1Z2} imply that $\Theta(K_\vep \hat  F_i) \cap \Theta(\hat F_j) = \emptyset$ for every $i\neq j$, and consequently
$$\inf\,\Big\{|u-v| \colon u \in \Theta(\hat  F_i), \,\, v \in \Theta(\hat  F_j)\Big\} \geqslant k_\vep.$$
To find $x \in \mathbb{T}^2$ as claimed, we are left to apply the specification property of $g_2$ (valid since $g_2$ is an Anosov diffeomorphism) and the equality \eqref{eqrelTheta}.

\smallskip

Consider now a tempered F{\o}lner sequence $(F_n)_{n\, \in \, \mathbb{N}}$ on $\mathbb Z^2$; for instance, the one defined by $F_n = [-n,n]^2 \subset \mathbb Z^2$. If $P$ is the common fixed point by both $g_1$ and $g_2$, then the probability measure $\mu_1 = \delta_P$ belongs to $\mathcal M_G(X)$, and the same happens with the Lebesgue measure (say $\mu_2$) on $\mathbb{T}^2$. Thus, given $\varphi \in C(\mathbb{T}^2, \mathbb R) \setminus \{0\}$ such that $\varphi \geqslant 0$ and $\varphi(P) = 0$, then $\int \varphi \, d\mu_1 = 0 \neq \int \varphi \, d\mu_2$, and therefore Corollary~\ref{cor:minimal1} asserts that
$$\Big\{x \in \mathbb{T}^2 \colon \,\lim_{n\,\to\,+\infty} \,\frac1{|F_n|} \,\sum_{g\,\in \,F_n} \,\varphi(g(x))\, \text{ does not exist} \Big\}$$
is a Baire generic subset of $\mathbb{T}^2$.

\end{example}

\noindent \textbf{Application 3}. \label{ex:free} A different setting concerns free Markov semigroups. Let $\Gamma$ be a free semigroup, finitely generated by $G =\{Id, g_1, g_2, \cdots, g_p\}$, where for every $1 \leqslant i \leqslant p$ the map $g_i: X \to X$ is a $\mu$-preserving transformation acting on a probability measure space $(X,\mathfrak{B}, \mu)$. The choice of $G$ endows $\Gamma$ with a norm $|\cdot|$ defined as follows: given $g \in \Gamma$, $|g|$ is the length of the shortest word over the alphabet $G$ representing $g$. Denote by $G_k$ the set $\{g \in \Gamma \colon \, |g| = k\}$. Now take $\varphi \in L^\infty(X, \mu)$ and consider the sequence of its spherical averages
$$k \in \mathbb{N} \quad \mapsto \quad s_k(\varphi) \,=\, \frac{1}{\# G_k} \,\sum_{g \, \in \, G_k}\, \varphi \circ g$$
where $\#$ stands for the cardinal of a finite set. 
Next consider the Ces\`aro averages of the previous spherical averages, that is,
\begin{equation}\label{eq:Cesaro-spherical}
n \in \mathbb{N} \quad \mapsto \quad \Phi_n \,=\, \frac{1}{n} \,\sum_{k=0}^{n-1}\,\frac{1}{\# G_k} \,\sum_{g \, \in \, G_k}\, \varphi \circ g.
\end{equation}
The main result of Bufetov \emph{et al} in \cite{BKK} establishes the pointwise convergence at $\mu$-almost everywhere of the sequence
$\big(\Phi_n\big)_{n  \, \in \, \mathbb{N}}$, as $n$ goes to $+\infty$, whenever $\Gamma$ is a Markov semigroup with respect to the generating set $G$. As a consequence of Theorem~\ref{CV_mod}, if there exists a dense set of points $x \in X$ such that $W_\Phi(x)$ is not a singleton, then the set of $\Phi$-irregular points is Baire generic in $X$. Let us check this information through an example.

\begin{example}\label{ex:z4-z6-free}
Consider the unit circle $\mathbb S^1=\{z \in \mathbb{C} \colon \, |z|=1\}$ and the self-maps of $\mathbb S^1$ given by $g_1(z) = z^4$ and $g_2(z)=z^6$. These transformations commute and have two fixed points in common, whose pre-orbits by $g_1$ and $g_2$ are dense in $\mathbb S^1$. Take the free (Markov) semigroup $G$ generated by $G_1=\{Id, g_1, g_2\}$, and let $\varphi$ be in $C(\mathbb S^1, \mathbb{R})$. It is known (cf. \cite{FV}) that, for a Baire generic subset of sequences $\omega = \omega_1 \omega_2 \cdots$ in $\{1, 2\}^\mathbb{N}$, the non-autonomous dynamics of $(g_{\omega_n} \circ \cdots \circ g_{\omega_1})_{n \, \in \, \mathbb{N}}$ has a Baire generic set of irregular points for the Ces\`aro averages
$$\varphi_n \,= \,\frac{1}{n} \,\sum_{j=0}^{n-1}\,\varphi \circ g_{\omega_j} \circ \cdots \circ g_{\omega_1}$$
of well chosen continuous potentials $\varphi$. Regarding the averages \eqref{eq:Cesaro-spherical} of $\varphi \in C(\mathbb S^1, \mathbb{R})$, in this case they are given by
$$\Phi_n \,= \,\frac{1}{n} \,\sum_{k=0}^{n-1}\,\frac{1}{2^k} \,\sum_{j=0}^k\, \binom kj \,(\varphi \circ g_1^j \circ g_2^{k-j}).$$
Let $z_0 \in \mathbb S^1$ be a common fixed point for $g_1$ and $g_2$. The sequence $\big(\Phi_n(z_0)\big)_{n \, \in \, \mathbb{N}}$ converges to $\varphi(z_0)$ as $n$ goes to $+\infty$. We claim that, for every $x$ in the pre-orbit $\mathcal O^-(z_0)$ of $z_0$ by the semigroup action (made up by the pre-images of $z_0$ by all the elements of the semigroup $G$), the sequence  $\big(\Phi_n(x)\big)_{n \, \in \, \mathbb{N}}$ converges to $\varphi(z_0)$ as well.

\smallskip

Given $x \in \mathcal O^-(z_0)$, there exists $g=g_{i_n}\circ \dots \circ g_{i_2}\circ g_{i_1}\in G$, where $i_j \in \{1,2\}$, such that $g(x)=z_0$. Yet, as $g_1$ and $g_2$ commute, one can simply write $g=g_1^{a} \circ g_2^{b}$ for some non-negative integers $a$ and $b$.
Assume that $a,b\geqslant 1$ (the remaining cases are identical). If $k \geqslant a+b$, the sum
\begin{equation}\label{eq.sum2k}
\frac{1}{2^k}  \,\sum_{j=0}^k\, \binom kj \,(\varphi \circ g_1^j \circ g_2^{k-j})(x)
\end{equation}
may be rewritten as
$$\frac{1}{2^k} \left[\sum_{j=0}^{a-1} \binom kj (\varphi \circ g_1^j \circ g_2^{k-j})(x) + \sum_{j=a}^{k-b} \binom kj (\varphi \circ g_1^{j} \circ g_2^{k-j})(x)
	+  \sum_{j=k-b+1}^{k} \binom kj (\varphi \circ g_1^j \circ g_2^{k-j})(x)\right].$$
The absolute values of the first and third terms in the previous sum are bounded above by
$$a\,\|\varphi\|_\infty\, \max\Big\{\binom k0, \cdots, \binom ka \Big\} \, 2^{-k}$$
and
$$b\,\|\varphi\|_\infty\, \max\Big\{\binom k{k-b+1}, \cdots, \binom kk \Big\} \, 2^{-k}$$
respectively, and both estimates converge to zero as $k$ goes to $+\infty$. Thus, their Ces\`aro averages also converge to $0$.
Regarding the middle term, as $g_1$ and $g_2$ commute and $z_0$ is fixed by $G$, one has
\begin{align*}
\frac{1}{2^k} & \,\sum_{j=a}^{k-b} \, \binom kj \,(\varphi \circ g_1^{j} \circ g_2^{k-j})(x) = \frac{1}{2^k} \,\sum_{j=a}^{k-b} \, \binom kj \,\big(\varphi \circ g_1^{j-a} \circ g_2^{k-j-b}\big)(g_1^a \circ g_2^b)(x) \\
& = \frac{1}{2^k} \,\sum_{j=a}^{k-b} \, \binom kj \,(\varphi \circ g_1^{j-a} \circ g_2^{k-j-b})(z_0) = \Big[ \frac{1}{2^k} \,\sum_{j=a}^{k-b} \, \binom kj \Big] \cdot
\varphi (z_0)
\end{align*}
whose limit as $k$ goes to $+\infty$ is precisely $\varphi(z_0)$. This proves that the sequence \eqref{eq.sum2k} converges to $\varphi(z_0)$, hence the same happens with its Ces\`aro averages.

\smallskip

Therefore, if $z_0 \in \mathbb S^1$ and $z_1 \in \mathbb S^1$ are the two common fixed points, $x \in \mathbb S^1$ belongs to the dense pre-orbit by the semigroup action of $z_0$ and $y \in \mathbb S^1$ belongs to the dense pre-orbit by the semigroup action of $z_1$, then
the sequence $\big(\Phi_n(x)\big)_n$ converges to $\varphi(z_0)$ and the sequence $\big(\Phi_n(y)\big)_n$ converges to $\varphi(z_1)$.
So, taking into account that $\mathcal O^-(z_0)$ is dense in $\mathbb S^1$ (since the pre-orbit by $g_1$ is dense), if we choose $\varphi\in C(\mathbb S^1,\mathbb R)$ such that $\varphi(z_0) \neq \varphi(z_1)$ then, by Theorem~\ref{CV_mod}, we conclude that ${\mathcal I}(\Phi)$ is Baire generic in $\mathbb S^1$.
\end{example}

\begin{remark}
As the generators of the semigroup action described in Example~\ref{ex:z4-z6-free} commute, one could consider, alternatively, the sequence
$(\Psi_n)_n$ where
$$
\Psi_n(\cdot) \,= \,\frac{1}{n^2} \,\sum_{k,\,\ell=0}^{n-1}\,\,(\varphi \circ g_1^k \circ g_2^{\ell})(\cdot).$$
If $z_0\in \mathbb S^1$ is a common fixed point for $G$ and $x \in \mathcal O^-(z_0)$, then this sequence can be rewritten as
$$
\frac{1}{n^2} \left[\sum_{(k,\,\ell)\,\in \,[a,n-1] \times [b,n-1]} (\varphi \circ g_1^\ell \circ g_2^{k-\ell})(x)
	+ \sum_{(k,\,\ell)\,\notin \,[a,n-1]\times [b,n-1]} (\varphi \circ g_1^\ell \circ g_2^{k-\ell})(x) \right]$$
where the sum is taken over pairs of integers $(k,\ell)$. The first term is equal to
$$
\frac{1}{n^2} \sum_{(k,\,\ell)\,\in \,[a,n-1]\times [b,n-1]} (\varphi \circ g_1^{\ell-a} \circ g_2^{k-b-\ell})(z_0) = \frac{(n-a)(n-b)}{n^2}\; \varphi (z_0)
$$
and converges to $\varphi (z_0)$ as $n$ goes to $+\infty$; meanwhile, the second term has absolute value bounded above by $\frac{a+b}n \, \|\varphi\|_\infty$, which goes to zero as $n$ tends to $+\infty$. Thus, if $z_0 \in \mathbb S^1$ and $z_1 \in \mathbb S^1$ are two common fixed points and $\varphi(z_0)\neq\varphi(z_1)$ then, by Theorem~\ref{CV_mod}, one concludes that ${\mathcal I}(\Psi)$ is Baire generic in $\mathbb S^1$.
\end{remark}

\subsection*{Acknowledgments}{LS is partially supported by Faperj, Fapesb-JCB0053/2013 and CNPq. MC and PV are supported by CMUP, which is financed by national funds through FCT-Funda\c c\~ao para a Ci\^encia e a Tecnologia, I.P., under the project UIDB/00144/2020. PV also acknowledges financial support from Funda\c c\~ao para a Ci\^encia e Tecnologia (FCT) - Portugal through the grant CEECIND/03721/2017 of the Stimulus of Scientific Employment, Individual Support 2017 Call.}

\def\cprime{$'$}


\end{document}